\documentclass[12pt]{article}
\usepackage[utf8]{inputenc}
\usepackage{amssymb,amsmath,amsfonts,amsthm}
\usepackage{enumerate}
\usepackage{adjustbox,tabu}
\usepackage{hyperref}
\newtheorem{lem}{Lemma}
\newtheorem{theorem}{Theorem}
\newtheorem{defi}{Definition}

\newtheorem{prop}{Proposition}
\newtheorem{rem}{Remark}

\newtheorem{qst}{Question}

\date{\vspace{-3em}}

\newcommand{\Addresses}{{
  \bigskip
  \footnotesize

  I.~Gorshkov, \textsc{Sobolev Institute of Mathematics, Novosibirsk, Russia;}\par\nopagebreak
  \textit{E-mail address: } \texttt{ilygor8@gmail.com}

  \medskip

  A.~Staroletov, \textsc{Sobolev Institute of Mathematics, Novosibirsk, Russia;}\par\nopagebreak
  \textsc{Novosibirsk State University, Novosibirsk, Russia;}\par\nopagebreak
  \textit{E-mail address: } \texttt{staroletov@math.nsc.ru}
}}

\title{On primitive $3$-generated axial algebras of Jordan type}

\author{Ilya Gorshkov and Alexey Staroletov\footnote{The work is supported by Mathematical Center in Akademgorodok under agreement No. 075-15-2019-1613 with the Ministry of Science and Higher Education of the Russian Federation.}}

\begin{document}
\maketitle
\begin{abstract}
Axial algebras of Jordan type $\eta$ are commutative algebras generated by 
idempotents whose adjoint operators have the minimal polynomial dividing $(x-1)x(x-\eta)$, where $\eta\not\in\{0,1\}$ is fixed, with restrictive multiplication rules. These properties generalize the Pierce decompositions for idempotents in Jordan algebras, where $\frac{1}{2}$ is replaced with $\eta$.
In particular, Jordan algebras generated by idempotents are 
axial algebras of Jordan type $\frac{1}{2}$. If $\eta\neq\frac{1}{2}$ then it is known that axial algebras of Jordan type $\eta$ are factors of the so-called Matsuo algebras corresponding to 3-transposition groups.

We call the generating idempotents {\it axes} and say that an axis is {\it primitive} if its adjoint operator has 1-dimensional 1-eigenspace.
It is known that a subalgebra generated by two primitive axes has dimension at most three. The 3-generated case has been opened so far. We prove that any axial algebra of Jordan type generated by three primitive axes has dimension at most nine. If the dimension is nine
and $\eta=\frac{1}{2}$ then we either show how to find a proper ideal in this algebra or prove that the algebra is isomorphic to certain Jordan matrix algebras.
\end{abstract}

{\bf keywords:} axial algebra, Jordan algebra, Matsuo algebra, idempotents, axis

\section{Introduction}

The concept of axial algebras was introduced by Hall, Rehren, and Shpectorov \cite{HRS2015}.
These algebras are commutative, nonassociative, and generated by idempotents.
Axial algebras extend the class of Majorana algebras which was introduced as a part of Majorana theory
by Ivanov \cite{Iv09}. The key motivational example for both theories is the Griess algebra, which is a real commutative nonassociative algebra of dimension 196884 that has the Monster sporadic simple group as its automorphism group. It is known that this algebra is generated by idempotents that correspond to involutions from the conjugacy class 2A of the Monster. 
Besides the class of Majorana algebras,  axial algebras also include Jordan algebras generated by idempotents and Matsuo algebras corresponding to 3-transposition groups. Thus, this class provides a common approach for studying seemingly different objects.

In this paper we focus on the class of axial algebras of Jordan type.
Consider a field $\mathbb{F}$ and a commutative $\mathbb{F}$-algebra $A$.
If $a\in A$ then by $ad_a$ we denote the right multiplication map on $A$, which maps $x\in A$
to $xa$, and by $A_\lambda(a)$ the eigenspace of $ad_a$ associated with $\lambda\in\mathbb{F}$. Suppose that $\eta$ is an element of $\mathbb{F}$ and $\eta\neq0,1$.
We say that $A$ is an {\it axial algebra} $A$ of Jordan type $\eta$, 
if $A$ is generated by a set of idempotents $X$ such that for every $a\in X$ we have a decomposition $A=A_1(a)\oplus  A_0(a)\oplus A_\eta(a)$, and the pairwise products of elements of $A_1(a)$, $A_0(a)$, and $A_\eta(a)$ satisfy certain fusion (i.e., multiplication) rules. These fusion rules imitate the Peirce multiplication rules in Jordan algebras and repeat them in the case $\eta=1/2$. 
It was proved in~\cite{HRS} (with a correction in~\cite{HSS2018-1}) that, for $\eta\neq1/2$, algebras of Jordan type are the Matsuo algebras, corresponding to 3-transposition groups, or their factor algebras. Therefore, the case $\eta=1/2$ is special for axial algebras of Jordan type and for this $\eta$ they are called {\it axial algebras of Jordan type half}. The class of Matsuo algebras was introduced by Matsuo~\cite{Matsuo} and later generalized in~\cite{HRS}. 

Idempotents of the generating set $X$ of an axial algebra $A$ are called {\it axes}. If $a\in X$ is an axis and $A_1(a)$ is 1-dimensional then
we say that $a$ is a {\it primitive axis}. If all axes of $X$ are primitive then $A$ is called a {\it primitive
axial algebra}. Subalgebras of primitive axial algebras of Jordan type generated by two axes
were classified in~\cite{HRS}. In particular, their dimension is at most three.
The 3-generated case was left open and stated as Problem~1(i) in \cite{HSS2018-1}.
In this paper we show that the dimension of a 3-generated primitive axial algebra of Jordan type does not exceed nine.
Moreover, in the case $\eta=1/2$ we build the universal 9-dimensional algebra, depending on some parameters, such that every 3-generated primitive axial algebra of Jordan type half is a factor algebra of this algebra for suitable values of the parameters. 

Consider a primitive axis $a$ of $A$ and an element $y$ of $A$. Denote by $\varphi_x(y)$ the element of $\mathbb{F}$ such that the projection of $y$ on $A_1(x)$ equals $\varphi_x(y)x$. The main result of this paper is the following statement.
\begin{theorem}\label{thm:1} Suppose that $\mathbb{F}$ is a field of characteristic not two and $\eta\in\mathbb{F}$, $\eta\neq0,1$.
If $A$ is an axial algebra of Jordan type $\eta$ over $\mathbb{F}$ generated by primitive axes $a$, $b$, and $c$ then $A$ is the span of $a$, $b$, $c$, $ab$, $bc$, $ac$, $(ab)c$, $(ac)b$, and $(bc)a$: in particular $\dim A\leq 9$. Furthermore,
if $\alpha=\varphi_a(b)$, $\beta=\varphi_b(c)$, $\gamma=\varphi_c(a)$, and $\psi=\varphi_a(bc)$ then
\begin{enumerate}[(i)]
 \item $A_0(a)=\langle ab-\alpha(1-\eta)a-\eta{b}, 
                       ac-\gamma(1-\eta)a-\eta{c}, 
                       a(bc)-\eta{bc}-\psi(1-\eta){a}, 
                       b(ac)+c(ab)-\eta{bc}-\eta\alpha{c}-\eta\gamma{b}-(2\alpha\gamma+\eta\psi-4\eta\alpha\gamma)a\rangle$;
 \item $A_\eta(a)=\langle  b(ac) - c(ab), ab - \alpha{a}, ac-\gamma{a}, a(bc)-\psi{a}\rangle$;
 \item either $\eta=\frac{1}{2}$ or $\alpha(\beta-\gamma)(\eta-2\alpha)=\beta(\gamma-\alpha)(\eta-2\beta)=\gamma(\alpha-\beta)(\eta-2\gamma)=~0$.
\end{enumerate}
\end{theorem}
We show in~Table~\ref{t:prod} how to multiply elements in the algebra from this theorem if $\eta=1/2$.
Moreover, in Section~\ref{sec:4} we prove that the algebra with such a table of products is indeed a primitive axial algebra of Jordan type half and, moreover, it is a Jordan algebra.

Suppose that $\mathbb{F}$ is a field of characteristic not two. Consider $M_n(\mathbb{F})$, the algebra of 
all square $n\times n$ matrices over $\mathbb{F}$. Define another product of matrices $X$ and $Y$ by $X\circ Y=\frac{1}{2}(XY+YX)$. It is well-known that the set of matrices with respect to this product forms a Jordan algebra which is denoted by $M_n(\mathbb{F})^+$. Moreover, if $j$ is an involution of $M_n(\mathbb{F})$ then the set of all
$j$-symmetric elements is closed under the Jordan product $\circ$, so we obtain a subalgebra of $M_n(\mathbb{F})^+$ which is denoted by $H(M_n(\mathbb{F}),j)$. 
Recall that an algebra is simple if it contains no non-trivial ideals and the multiplication operation is not zero.
\begin{theorem}\label{thm:2} Consider an axial algebra $A$ of Jordan type half over a field $\mathbb{F}$ with $char(\mathbb{F})\neq2$. Suppose that $A$ is generated by primitive axes $a$, $b$, and $c$ and denote
$\alpha=\varphi_a(b)$, $\beta=\varphi_b(c)$, $\gamma=\varphi_c(a)$, and $\psi=\varphi_a(bc)$.
If $A$ is 9-dimensional then the following statements hold.
\begin{enumerate}[(i)]
\item If $(\alpha+\beta+\gamma-2\psi-1)(\alpha\beta\gamma-\psi^2)\neq0$ and $\psi^2-\alpha\beta\gamma$ is a square in $\mathbb{F}$ then $A$ is isomorphic to $M_3(\mathbb{F})^+$.
\item If $(\alpha+\beta+\gamma-2\psi-1)(\alpha\beta\gamma-\psi^2)\neq0$ and $\psi^2-\alpha\beta\gamma$ is not a square in $\mathbb{F}$ then let $\mathbb{P}=\mathbb{F}(z)$, where $z^2=\psi^2-\alpha\beta\gamma$, and let the bar denote the involution of $\mathbb{P}$ given by $x+zy\mapsto x-zy$ for $x,y\in\mathbb{F}$.
Then $A$ is isomorphic to $H(M_3(\mathbb{P}),j)$, where $j$ is the involution given by $X^j=T^{-1}\overline{X}'T$ with some $T\in M_3(\mathbb{P})$ such that $T^j=tT$, $t\in\mathbb{P}$ and $t\overline{t}=1$.
\item If $(\alpha+\beta+\gamma-2\psi-1)(\alpha\beta\gamma-\psi^2)=0$
then $A$ is not simple.
\end{enumerate}
\end{theorem}
\begin{rem} For the claim $(ii)$, there may be several non-isomorphic algebras. For example, if $\mathbb{F}=\mathbb{R}$ is the field of real numbers then there are two different algebras and both occur \cite[Exercise~6, p.211]{Ja68}.
\end{rem}

Since we prove that for $\eta=1/2$ three primitive axes generate a Jordan algebra,
the following natural question arises.
\begin{qst} Is it true that every primitive axial algebra of Jordan type is either
a factor algebra of a Matsuo algebra or a Jordan algebra?
\end{qst}

This paper is organized as follows. In Section~2 we give basic definitions on axial and Jordan algebras. In Section~3 we prove Theorem~\ref{thm:1}. Finally, Section~4 is devoted to different properties of axial algebras of Jordan type half generated by three primitive axes, there we also prove Theorem~\ref{thm:2}.

\section{Preliminaries}\label{sec:2}

In this section we provide basic definitions, introduce notation, and describe properties of axial algebras of Jordan type. Recall that we do not assume associativity of the product in  algebras.
\subsection{Axial algebras of Jordan type and related objects}

First, we recall the definition of Jordan algebras that are closely related to the objects considered in this paper. We follow~\cite{Ja68} on this topic.
Throughout, we suppose that $\mathbb{F}$ is a field of characteristic not two. 
\begin{defi} A commutative $\mathbb{F}$-algebra is called a Jordan algebra
if its elements satisfy the identity $(x^2y)x=x^2(yx)$.
\end{defi}

Starting from an associative algebra one can construct a Jordan algebra in the following way.

\begin{defi} Consider an associative $\mathbb{F}$-algebra $(A, \cdot, +)$ and defines on $A$ the Jordan product $x \circ y = \cfrac{1}{2}(xy+yx)$. Then the algebra $A^+=(A,\circ,+)$ is a Jordan algebra. 
\end{defi}

Another way to obtain a Jordan algebra is to use an involution of an associative algebra.
\begin{defi} Consider an associative $\mathbb{F}$-algebra $A$.
Then any $\mathbb{F}$-endomorphism $j$ of $A$ such that 
$(xy)^j=y^jx^j$ and $x^{j^2}=x$ for every $x,y\in A$ is called an involution of $A$.
\end{defi}

\begin{defi} Consider an associative algebra $A$ with an involution $j$.
Then by $H(A,j)$ we denote the subalgebra of $A^+$ comprising of all $j$-symmetric elements in the sense that $x^j=x$.
\end{defi}

 The class of axial algebras was introduced by Hall, Rehren, and Shpectorov \cite{HRS2015}. These algebras are commutative, nonassociative, and generated by idempotents
 with adjoint action semi-simple. For more detailed information about this class one can see recent papers~\cite{kms} and ~\cite{GVMSS}. In the present paper we concentrate on the so-called axial algebras of Jordan type that were introduced in \cite{HRS}.

It is known that for any idempotent of a Jordan algebra eigenvalues of its adjoint operator belong in the set $\{1,0,\frac{1}{2}\}$ and the algebra is the direct sum of the corresponding eigenspaces~\cite{Ja68}. This decomposition of a Jordan algebra into the direct sum of eigenspaces is called {\it the Peirce decomposition}.
The axial algebras of Jordan type have similar decompositions. If $A$ is an $\mathbb{F}$-algebra then denote by $A_\lambda(a)$ the eigenspace of $ad_a$ associated with $\lambda\in\mathbb{F}$.

\begin{defi}~\label{d:jtype} Axial algebras of Jordan type $\eta$ over $\mathbb{F}$, where $\eta\not\in\{0,1\}$ is fixed, 
are algebras generated by a set of idempotents $X$ such that for each $a\in X$:

(1) the minimal ad-polynomial of $a$ divides $(x-1)x(x-\eta)$;

(2) the fusion rules for products of eigenvectors of $ad_a$ are the following: 
$$A_1(a)A_1(a)\subseteq A_1(a)\text{ and }A_0(a)A_0(a)\subseteq A_0(a), A_1(a)A_0(a)=\{0\},$$ 
$$(A_0(a)+A_1(a))A_{\eta}(a)\subseteq A_{\eta}(a)\\ \text{, and }A_\eta(a)^2\subseteq A_0(a)+A_1(a).$$
\end{defi}
Note that conditions $(1)$ and $(2)$ for $\eta=1/2$ hold for any idempotent in a Jordan algebra.

\begin{defi} An element $a\in X$ is called an axis. If $A_1(a)$ is 1-dimensional, i.e. coincides with $\langle a \rangle$, then $a$ is a {\it primitive axis}.
The algebra $A$ is primitive if all elements of $X$ are primitive.
\end{defi}

The following lemma shows that the associative law holds for certain elements of axial algebras of Jordan type. 
\begin{lem}(Seress Lemma, \cite[Lemma~4.3]{HRS})\label{l:seress}. 
Consider an axial algebra $A$ of Jordan type. If $a$ is a primitive axis of $A$, $x\in A$, and $y\in A_1(a)+A_0(a)$ then $a(xy)=(ax)y$.
\end{lem}
We widely use this lemma in our proofs later on.

\subsection{Frobenius form}

It turned out that in the majority of known examples primitive axial algebras admit a bilinear form that associates with the algebra product. 

\begin{defi} A non-zero bilinear form (·,·) on a commutative algebra $A$
is called Frobenius if the form associates with the algebra product, that is, $(ab,c) = (a,bc)$ for all $a,b,c\in A$.
\end{defi}

In case of primitive axial algebras of Jordan type such a from can be defined using the projection map on 1-dimensional 1-eigenspaces of adjoint operators.

\begin{defi} Suppose that $a$ is a primitive axis of an axial algebra $A$.
Then for every $b\in A$ by $\varphi_a(b)$ we denote the element of $\mathbb{F}$
such that the projection of $b$ on $A_1(a)$ equals $\varphi_a(b)a$.
\end{defi}

Now we are ready to formulate an important result stating that every primitive axial algebra of Jordan type does admit a Frobenius form.

\begin{lem}\cite[Theorem~4.1]{HSS2018}\label{l:frobenius}
Suppose that $A$ be a primitive axial algebra of Jordan type $\eta$ generated by the set of primitive axes $X$. 
There exists a unique Frobenius bilinear form on $A$ such that
\begin{enumerate}
\item  $(a,u)=\varphi_a(u)$, for all $a\in X$ and all $u\in A$;
\item $(a,a)=1$, for all $a\in X$;
\item (·,·) is invariant under automorphisms of $A$.
 \end{enumerate}
 Moreover, this form is symmetric.
\end{lem}

From now on, when we consider a Frobenius form on a primitive axial algebra of Jordan type with a given generating set, 
we assume that this form is from the conclusion of Lemma~\ref{l:frobenius}.
Now we prove some auxiliary results using values of the form.
\begin{lem}\label{l:2gen}
Suppose that $A$ is an axial algebra of Jordan type.
If $a$ is a primitive axis in $A$ and $b$ is an element of $A$ then
\begin{enumerate}[(i)]
 \item $ab-(a,b)a\in A_{\eta}(a)$;
 \item $ab-(a,b)(1-\eta)a-\eta{b}\in A_0(a)$.
\end{enumerate}
\begin{proof}
Write $b=(a,b)a+b_0+b_\eta$, where $b_0\in A_0(a)$ and $b_\eta\in A_\eta(a)$.
Then $ab=(a,b)a+\eta b_\eta$. Now $ab-(a,b)a=\eta b_\eta\in A_\eta(a)$
and $ab-\eta{b}=(a,b)(1-\eta)a-\eta b_0$. Therefore, we have $ab-\eta{b}-(a,b)(1-\eta)a\in A_0(a)$ and the lemma follows.
\end{proof}

\end{lem}

\begin{lem}\label{l:cab} Suppose that $A$ is a primitive axial algebra of Jordan type generated by a set of primitive axes $X$ and $(\cdot,\cdot)$ is the Frobenius form on $A$. Assume that $a\in X$ and $b,c\in A$.
Denote $\alpha=(a,b)$, $\gamma=(a,c)$, and $\psi=(a,bc)$.
Then $(a,b(ac))=(a,c(ab))=(1-\eta)\alpha\gamma+\eta\psi$.
\end{lem}
\begin{proof} First, we see that $(a,b(ac))=(ba,ac)=(ca,ab)=(a,c(ab))$.
So it suffices to find $(a,c(ab))$.
Write $b=\alpha a+b_0+b_\eta$ and  $c=\gamma a+c_0+c_\eta$, where $b_0, c_0\in A_0(a)$ and $b_\eta, c_\eta\in A_\eta(a)$. Then	
$$bc=(\alpha\gamma\cdot{a}+b_0c_0+b_\eta c_\eta)+(\alpha\eta\cdot c_\eta+\gamma\eta\cdot b_\eta+b_0c_\eta+c_0b_\eta).$$	
It follows from the fusion rules of $A$ that $\alpha\eta\cdot c_\eta+\gamma\eta\cdot b_\eta+b_0c_\eta+c_0b_\eta\in A_\eta(a)$ and $b_0c_0\in A_0(\eta)$, so $\psi=(a,bc)=(a,\alpha\gamma\cdot{a}+b_\eta c_\eta)=\alpha\gamma+(a,b_\eta c_\eta)$. Hence	$$(a,b_\eta c_\eta)=\psi-\alpha\gamma.$$

Now $ab=\alpha\cdot{a}+\eta\cdot{b_\eta}$ and so $$c(ab)=(\alpha\gamma\cdot{a}+\eta\cdot{b_\eta}c_\eta)+\eta(\alpha\cdot c_\eta+\gamma\eta\cdot b_\eta+b_\eta c_0).$$
Since $\eta(\alpha\cdot c_\eta+\gamma\eta\cdot b_\eta+b_\eta c_0)\in A_\eta(a)$, we have $(a,c(ab))=(a,\alpha\gamma\cdot{a}+\eta\cdot{b_\eta}c_\eta)=
\alpha\gamma+\eta(a,{b_\eta}c_\eta)=\alpha\gamma+\eta(\psi-\alpha\gamma)=(1-\eta)\alpha\gamma+\eta\psi$, as required.
\end{proof}

Extending the results of Lemma~\ref{l:2gen}, we show how to find eigenvectors of $ad_a$
involving other elements of the algebra.

\begin{lem}\label{l:A0-Aeta} Let $a$ be a primitive axis in $A$ and $b,c\in A$. 
Denote $\alpha=\varphi_a(b)$, $\gamma=\varphi_a(c)$, and $\psi=\varphi_a(bc)$.
Then the following statements hold.
\begin{enumerate}[(i)]
	\item $\langle ab-\alpha(1-\eta)a-\eta{b}, 
	ac-\gamma(1-\eta)a-\eta{c}, 
	a(bc)-\eta{bc}-\psi(1-\eta){a}, 
	b(ac)+c(ab)-\eta{bc}-\eta\alpha{c}-\eta\gamma{b}-(2\alpha\gamma+\eta\psi-4\eta\alpha\gamma)a\rangle\subseteq A_0(a)$;
	\item $\langle  b(ac) - c(ab), ab - \alpha{a}, ac-\gamma{a}, a(bc)-\psi{a}\rangle\subseteq A_\eta(a)$.
\end{enumerate}
\end{lem}

\begin{proof} Lemma~\ref{l:2gen} implies that 
$ab-\alpha(1-\eta)a-\eta{b}$, $ac-\gamma(1-\eta)a-\eta{c}$, $a(bc)-\eta{bc}-\psi(1-\eta){a}$ belong in $A_0(a)$ and $ab - \alpha{a}$, $ac-\gamma{a}$, $a(bc)-\psi{a}$ belong in $A_\eta(a)$.
Now we prove that $b(ac)+c(ab)-\eta{bc}-\eta\alpha{c}-\eta\gamma{b}-(2\alpha\gamma+\eta\varphi-4\eta\alpha\gamma)a$ is in $A_0(a)$ and $a(bc)-\psi{a}$ is in $A_\eta(a)$.
Write $b=\alpha a+b_0+b_\eta$ and  $c=\gamma a+c_0+c_\eta$, where $b_0, c_0\in A_0(a)$ and $b_\eta, c_\eta\in A_\eta(a)$.
As in the proof of Lemma~\ref{l:cab}, we have $(ab)c=(\alpha\gamma\cdot{a}+\eta\cdot{b_\eta}c_\eta)+\eta(\alpha\cdot c_\eta+\gamma\eta\cdot b_\eta+b_\eta c_0)$.
Combining with the corresponding expression for $(ac)b$, we find that
$$(ab)c+(ac)b=(2\alpha\gamma\cdot{a}+2\eta\cdot{b_\eta}c_\eta)+\eta(\gamma\cdot b_\eta+\alpha\cdot c_\eta+\alpha\eta\cdot c_\eta+\gamma\eta\cdot b_\eta+b_\eta c_0+c_\eta b_0).$$

Therefore,  $(ab)c+(ac)b-\eta\cdot{bc}=((2\alpha\gamma-\eta\alpha\gamma)a+\eta\cdot{b_\eta}c_\eta-\eta\cdot b_0c_0)+\eta(\gamma\cdot b_\eta+ \alpha\cdot c_\eta)$.
Since $b-b_\eta\in A_0(a)+A_1(a)$ and $c-c_\eta\in A_0(a)+A_1(a)$, we infer that $$(ab)c+(ac)b-\eta\cdot{bc}-\eta\gamma\cdot {b}-\eta\alpha\cdot{c}\in A_0(a)+A_1(a).$$ 
Denote $x=(ab)c+(ac)b-\eta\cdot{bc}-\eta\gamma\cdot {b}-\eta\alpha\cdot{c}$.
Then clearly $x-\varphi_a(x)a\in A_0(a)$. So it remains to find $\varphi_a(x)$.
By the definition, we have
\begin{multline*}
\varphi_a(x)=(x,a)=((ab)c+(ac)b-\eta\cdot{bc}-\eta\gamma\cdot {b}-\eta\alpha\cdot{c},a)\\=((ab)c,a)+((ac)b,a)-\eta({bc},a)-\eta\gamma\cdot (b,a)-\eta\alpha(c,a)\\=(ab,ac)+(ac,ab)-\eta\psi-2\eta\gamma\alpha=2(c(ab),a)-\eta\psi-2\eta\gamma\alpha. 
\end{multline*}
By Lemma~\ref{l:cab},
we have $(a,c(ab))=(1-\eta)\alpha\gamma+\eta\psi$, so 
$$(x,a)=2((1-\eta)\alpha\gamma+\eta\psi)-\eta\psi-2\eta\gamma\alpha=2\alpha\gamma-4\eta\alpha\gamma+\eta\psi.$$ 
Thus $b(ac)+c(ab)-\eta{bc}-\eta\alpha{c}-\eta\gamma{b}-(2\alpha\gamma+\eta\psi-4\eta\alpha\gamma)a\in A_0(a)$.

Using the above expressions for $(ab)c$ and $(ac)b$, we obtain
$$(ab)c-(ac)b=\eta(\alpha\cdot c_\eta-\gamma\cdot b_\eta+\gamma\eta\cdot b_\eta-\alpha\eta\cdot c_\eta+b_\eta c_0-c_\eta b_0)$$
which is an element of $A_\eta(a)$ by the fusion rules of $A$.
\end{proof}

\subsection{Ideals}

To describe the variety of algebras, one needs to know how to find the ideals in order to go over to factor algebras. In this subsection we discuss ideals in axial algebras.
Methods to find all ideals in an axial algebra $A$ were developed in \cite{kms}.
We adapted those results to axial algebras of Jordan type.

\begin{defi} \label{projection graph}
Suppose that $A$ is a primitive axial algebra of Jordan type and $(\cdot,\cdot)$ is the Frobenius form as in Lemma~\ref{l:frobenius}. The \emph{projection graph} $\Delta$ of $A$ 
has as vertices all generating axes of $A$ and $\Delta$ has an edge between $a$ and $b$ 
if and only if $(a,b)\neq0$.
\end{defi} 

There is an ideal associated with the Frobenius form.
\begin{defi} The radical $A^\perp$ of the form is the set of elements of $A$ orthogonal 
to all elements of $A$ with respect to this form, that is
$$A^\perp=\{u\in A\mid (u,v)=0\mbox{ for all }v\in A\}.$$
\end{defi}

Recall that an algebra is simple if it contains no non-trivial ideals and the multiplication operation is not zero. It turns out that if the projection graph of $A$ is connected then
all proper ideals of $A$ are in the radical of the form.

\begin{prop}[\cite{kms}]\label{p:summary}
Suppose that $A$ is a primitive axial algebra of Jordan type and $(\cdot,\cdot)$ is the Frobenius form as in Lemma~\ref{l:frobenius}. If the projection graph $\Delta$ of 
$A$ is connected then $A$ is simple if and only if the Frobenius form has zero 
radical.  
\end{prop}

\section{Proof of Theorem~\ref{thm:1}}\label{sec:3}
The aim of this section is to prove Theorem~\ref{thm:1}.
Consider an axial algebra $A$ of Jordan type $\eta$. Suppose that $A$ is generated by primitive axes $a$, $b$, and $c$. As above, we assume that $(\cdot,\cdot)$ is the Frobenius form on $A$
such that $(a,a)=(b,b)=(c,c)=1$. Denote by $\mathcal{B}$ the set $\{a, b, c, ab, ac, bc, (ab)c, (ac)b, (bc)a\}$. We also use the following notation for values of the form:
$$\alpha=\varphi_a(b)=(a,b), \beta=\varphi_b(c)=(b,c), \gamma=\varphi_c(a)=(a,c), \psi=\varphi_a(bc)=(a,bc).$$
By Lemma~\ref{l:A0-Aeta}, the theorem will follow from the following statement and Lemma~\ref{l:relations}.

\begin{prop}\label{p:span} The algebra $A$ is the span of $\mathcal{B}$.
Moreover, if $\eta=1/2$ then pairwise products of elements of $\mathcal{B}$ can be found according to Table~\ref{t:prod}.
\end{prop}
\begin{rem} It is also possible to obtain the table of products for $\eta\neq1/2$ following the proofs of Lemmas~\ref{l:abc}, \ref{l:ab-ac-bc}, and \ref{l:abc-bac-cab},
but expressions are too big to write them explicitly. One can find the code for calculations on GitHub \cite{Git}.
\end{rem}
Since $\mathcal{B}$ contains $a$, $b$, and $c$, it suffices to show that the products of elements of $\mathcal{B}$ lie in the span $B$ of $\mathcal{B}$. Throughout, we suppose that $b=\alpha a+b_0+b_\eta$ and $c=\gamma a+c_0+c_\eta$,
where $b_0,c_0\in A_{0}(a)$ and $b_\eta,c_\eta\in A_{\eta}(a)$.
We split the proof of the proposition into several lemmas.

\begin{lem}\label{l:abc} If $x\in{B}$ then $ax$, $bx$, and $cx$ belong in $B$. 
\end{lem}
\begin{proof} We can assume that $x\in\mathcal{B}$ and consider products $ax$. 
Since  $aa$, $ab$, $ac$, $a(bc)\in\mathcal{B}$, it suffices to verify the cases $x\in\{ab,ac, (ab)c, (ac)b, (bc)a\}$.

Suppose that $x=ab$ or $x=ac$. Both cases are similar, therefore we assume that $x=ab$.  
It is easy to see that $ab=\alpha a+\eta b_\eta$.
Now we find that $$a(ab)=\alpha a+\eta^2b_\eta=(\alpha-\eta\alpha)a+\eta ab\in B.$$
 
Suppose that $x=(ac)b$ and $y=(ab)c$.  By Lemma~\ref{l:A0-Aeta}, we have $ax-ay=\eta{x}-\eta{y}\in B$ and $x+y-\eta{bc}-\eta\alpha{c}-\eta\gamma{b}-(2\alpha\gamma+\eta\psi-4\eta\alpha\gamma)a\in A_0(a)$.
So $ax+ay=\eta{a}(bc)+\eta\alpha{ac}+\eta\gamma{ab}+(2\alpha\gamma+\eta\psi-4\eta\alpha\gamma)a$.
Since
 $ax=(\frac{ax+ay}{2}+\frac{ax-ay}{2})$, we have 
 
 $$a((ac)b)=(\alpha\gamma+\frac{\eta\psi}{2}-2\eta\alpha\gamma)a+\frac{\eta\gamma}{2}{ab}+\frac{\eta\alpha}{2}{ac}+\frac{\eta}{2}a(bc)+\frac{\eta}{2}b(ac)-\frac{\eta}{2}c(ab).$$ 
Thus, $ax$ (and, similarly, $ay$) is in $B$.

Finally, consider the case $x=a(bc)$. Lemma~\ref{l:A0-Aeta} implies that $x-\eta{bc}-\psi(1-\eta)a\in A_0(a)$, so 

$$a(a(bc))=\psi(1-\eta)a+\eta{a(bc)}\in B.$$

Substituting $\eta=1/2$ in these expressions gives the suitable entries in Table~\ref{t:prod}.
\end{proof}

\begin{lem}\label{l:relations} It is true that either $\eta=1/2$ or 
$$\alpha(\beta-\gamma)(\eta-2\alpha)=\beta(\gamma-\alpha)(\eta-2\beta)=\gamma(\alpha-\beta)(\eta-2\gamma)=~0.$$
\end{lem}
\begin{proof}
Note that $(a,b(c(ab)))=(ab,c(ab))=(b,a(c(ab)))$. Now we find both 
expressions $(a,b(c(ab)))$ and $(b,a(c(ab)))$.
As in Lemma~\ref{l:abc}, we have

$$b(c(ab))=(\alpha\beta+\frac{\eta\psi}{2}-2\eta\alpha\beta)b+\frac{\eta\beta}{2}{ab}+\frac{\eta\alpha}{2}{bc}-\frac{\eta}{2}a(bc)+\frac{\eta}{2}b(ac)+\frac{\eta}{2}c(ab),$$ 
$$a(c(ab))=(\alpha\gamma+\frac{\eta\psi}{2}-2\eta\alpha\gamma)a+\frac{\eta\gamma}{2}{ab}+\frac{\eta\alpha}{2}{ac}+\frac{\eta}{2}a(bc)-\frac{\eta}{2}b(ac)+\frac{\eta}{2}c(ab).$$
Then 
\begin{multline*}
(a,b(c(ab)))=\\=(\alpha\beta+\frac{\eta\psi}{2}-2\eta\alpha\beta)(a,b)+\frac{\eta\beta}{2}(a,ab)+\frac{\eta\alpha}{2}(a,bc)-\frac{\eta}{2}(a,a(bc))+\frac{\eta}{2}(a,b(ac))+\frac{\eta}{2}(a,c(ab)).
\end{multline*}
We know that $\alpha=(a,b)=(a,ab)$ and $\psi=(a,bc)=(a,a(bc))$. Lemma~\ref{l:cab} implies that
$(a,b(ac))=(a,c(ab))=(1-\eta)\alpha\gamma+\eta\psi$, so
\begin{multline*}
(a,b(c(ab)))=\alpha\cdot(\alpha\beta+\frac{\eta\psi}{2}-2\eta\alpha\beta+\frac{\eta\beta}{2})+\psi\cdot(\frac{\eta\alpha}{2}-\frac{\eta}{2})+((1-\eta)\alpha\gamma+\eta\psi)\cdot(\frac{\eta}{2}+\frac{\eta}{2})=\\=-\eta^2\alpha\gamma-2\eta\alpha^2\beta+\eta^2\psi+\frac{\eta}{2}\alpha\beta+\eta\alpha\gamma+\eta\alpha\psi+\alpha^2\beta-\frac{\eta}{2}\psi.
\end{multline*}
Similarly, we have
\begin{multline*}
(b,a(c(ab)))=\alpha\cdot(\alpha\gamma+\frac{\eta\psi}{2}-2\eta\alpha\gamma+\frac{\eta\gamma}{2})+\psi\cdot(\frac{\eta\alpha}{2}-\frac{\eta}{2})+((1-\eta)\alpha\beta+\eta\psi)\cdot(\frac{\eta}{2}+\frac{\eta}{2})=\\
=-\eta^2\alpha\beta-2\eta\alpha^2\gamma+\eta^2\psi+\eta\alpha\beta+\frac{\eta}{2}\alpha\gamma+\eta\alpha\psi+\alpha^2\gamma-\frac{\eta}{2}\psi.
\end{multline*}
Therefore, we infer that
$0=(b,a(c(ab)))-(a,b(c(ab)))=
-\eta^2\alpha\beta+\eta^2\alpha\gamma+2\eta\alpha^2\beta-2\eta\alpha^2\gamma+\frac{\eta}{2}\alpha\beta-\frac{\eta}{2}\alpha\gamma-\alpha^2\beta+\alpha^2\gamma=\alpha(\gamma-\beta)(\eta-\frac{1}{2})(\eta-2\alpha)$.
Similarly, we obtain
$$\gamma(\beta-\alpha)(\eta-\frac{1}{2})(\eta-2\gamma)=\beta(\alpha-\gamma)(\eta-\frac{1}{2})(\eta-2\beta)=0.$$
\end{proof}

\begin{lem}\label{l:ab-ac-bc} If $x\in{B}$ then $(ab)x$, $(bc)x$ and $(ac)x$ belong in $B$. 
Moreover, if $x\in\mathcal{B}$ and $\eta=1/2$ then these products are as in Table~\ref{t:prod}.
\end{lem}
\begin{proof} Clearly, we can assume that $x\in\mathcal{B}$.
The cases for $(ab)x$, $(ac)x$, and $(bc)x$ are symmetric and, therefore, we consider only the products $(ab)x$.
The cases $x\in\{a,b,c\}$ are considered in Lemma~\ref{l:abc}, so we suppose that $x\in\{ab,bc,ac,(ab)c,(ac)b,(bc)a\}$.

Assume first that $x=ab$. We have $x(ab)=(ab)(ab)=(ab)(ab-\eta{b})+\eta{b}(ab)$. 
It follows from Lemma~\ref{l:2gen} that $ab-\eta{b}\in A_0(a)+A_1(a)$. Lemma~\ref{l:seress} yields $(ab)(ab-\eta{b})=a(b(ab)-\eta{b})$ and hence $x(ab)\in B$ by Lemma~\ref{l:abc}. 
Now we find $x(ab)$ for $\eta=1/2$.
By Lemma~\ref{l:abc}, we have $b(ab)=\frac{1}{2}(\alpha{b}+{ab})$ and $a(ab)=\frac{1}{2}(\alpha{a}+{ab})$. Hence,
\begin{multline*}
(ab)(ab)=a(b(ab-\frac{1}{2}{b}))+\frac{1}{2}{b}(ab)=
a(\frac{(\alpha-1)}{2}b+\frac{1}{2}ab)+\frac{1}{4}(\alpha{b}+{ab})=\\=\frac{(\alpha-1)}{2}ab+\frac{1}{4}(\alpha{a}+ab)+\frac{1}{4}(\alpha{b}+{ab})=\frac{\alpha}{4}(a+b+2ab)
.\end{multline*}

Suppose that $x=bc$ or $x=ac$. Clearly, both cases are similar, so we can assume that $x=bc$. By Lemma~\ref{l:2gen}, $bc-\eta{c}\in A_0(b)+A_1(b)$ and hence $(ab)x=(ab)(bc)=(ba)(bc-\eta{c})+\eta{c}(ab)=b(a(bc)-\eta{ac})+\eta{c}(ab)=b(a(bc))-\eta{b(ac)}+\eta{c(ab)}$ which is an element of $B$ by Lemma~\ref{l:abc}. 
Let now $\eta=\frac{1}{2}$. Then we know that $b(a(bc))=\frac{1}{4}(\psi{b}+\beta{ab}+\alpha{bc}+a(bc)+b(ac)-c(ab)).$

This implies that
\begin{multline*}
(ab)(bc)=\frac{1}{4}(\psi{b}+\beta{ab}+\alpha{bc}+a(bc)+b(ac)-c(ab))-\frac{1}{2}{b(ac)}+\frac{1}{2}{c(ab)}=\\=
\frac{1}{4}(\psi{b}+\beta{ab}+\alpha{bc}+a(bc)-b(ac)+c(ab)).
\end{multline*}

Suppose that $x=(ac)b$ and $y=(ab)c$. Applying Lemma~\ref{l:seress}, we see that
\begin{multline*}
(ab)x=(ba)((ac)b-\eta{ac}+\eta{ac})=(ba)((ac)b-\eta{ac})+\eta(ab)(ac)=\\=
b(a((ac)b-\eta{ac}))+\eta(ab)(ac)=b\big(a(b(ac))-\eta{a}(ac)\big)+\eta(ab)(ac).
\end{multline*}
Since $(ab)\cdot(ac)\in B$, we have $(ab)x\in B$. Now we find $(ab)x$ for $\eta=\frac{1}{2}$.
Lemma~\ref{l:abc} implies that 
\begin{multline*}
a(b(ac)-\frac{1}{2}(ac))=a(b(ac))-\frac{1}{2}{a}(ac))=\frac{1}{4}\big(\psi{a}+\gamma{ab}+\alpha{ac}+a(bc)+b(ac)-c(ab)\big)-\\-\frac{1}{4}(\gamma{a}+ac)=
\frac{1}{4}\big((\psi-\gamma){a}+\gamma{ab}+(\alpha-1)ac+a(bc)+b(ac)-c(ab)\big).
\end{multline*}
Then
\begin{multline*}
b\big(a(b(ac))-\frac{1}{2}(ac)\big)=\frac{1}{4}\Big((\psi-\gamma){ab}+\gamma{b(ab)}+(\alpha-1)b(ac)+b\big(a(bc)+b(ac)-c(ab)\big)\Big).
\end{multline*}
We know that $a(bc)-c(ab)\in A_{1/2}(b)$ and hence $b(a(bc)-c(ab))=\frac{1}{2}(a(bc)-c(ab))$.
Moreover, Lemma~\ref{l:abc} implies that $b(ab)=\frac{1}{2}(\alpha{b}+ab)$ and $b(b(ac))=\frac{1}{2}(\psi{b}+b(ac))$.
Therefore 
$$b\big(a(b(ac))-\frac{1}{2}(ac)\big)=\frac{1}{8}\big((\alpha\gamma+\psi)b+(2\psi-\gamma)ab+a(bc)+(2\alpha-1)b(ac)-c(ab)\big).$$
To find $(ab)x$, it remains to add $\frac{1}{2}(ab)(ac)$ which was considered above.
Finally, we have
$$ (ab)x=(ab)(b(ac))=\frac{1}{8}\big(\psi{a}+(\alpha\gamma+\psi)b+2\psi{ab}+\alpha{ac}+2\alpha{b(ac)}\big).$$

Now we find $(ab)y$. Lemma~\ref{l:A0-Aeta} implies that
$x+y-\eta{bc}-\eta\gamma{b}-\eta\alpha{c}\in A_{1}(a)+A_{0}(a)$.
It follows from Lemma~\ref{l:seress} that $(ab)(x+y-\eta{bc}-\eta\gamma{b}-\eta\alpha{c})=a\big(b(x+y-\eta{bc}-\eta\gamma{b}-\eta\alpha{c})\big)$ which is an element of $B$ by Lemma~\ref{l:abc}. 
Since
$$(ab)y=a\big(b(x+y-\eta{bc}-\eta\gamma{b}-\eta\alpha{c})\big)-(ab)x+\eta(ab)(bc+\gamma{b}+\alpha{c})$$
and $(ab)x$, $ab\cdot bc\in B$, we infer that $(ab)y\in B$. 
Now we find $(ab)y$ for $\eta=\frac{1}{2}$ using this expression.
By Lemma~\ref{l:abc}, we have 
$$bx=b(b(ac))=\frac{1}{2}({\psi}b+b(ac)),by=b(c(ab))=\frac{1}{4}\big({\psi}b+\beta{ab}+\alpha{bc}-a(bc)+b(ac)+c(ab)\big)$$
$$b(bc)=\frac{\beta}{2}b+\frac{bc}{2}.$$
Therefore, 
\begin{multline*}
b(x+y-\frac{1}{2}{bc}-\frac{\gamma}{2}{b}-\frac{\alpha}{2}{c})=
(\frac{\psi}{2}b+\frac{1}{2}b(ac))+\frac{1}{4}\big({\psi}b+\beta{ab}+\alpha{bc}-a(bc)+b(ac)+c(ab)\big)-\\-(\frac{\beta}{4}b+\frac{bc}{4})-\frac{\gamma}{2}{b}-\frac{\alpha}{2}{bc}=
\frac{1}{4}\big((3\psi-\beta-2\gamma)b+\beta{ab}+(-\alpha-1)bc-a(bc)+3b(ac)+c(ab)\big).
\end{multline*}
This implies that
\begin{multline*}
a(b(x+y-\frac{1}{2}{bc}-\frac{\gamma}{2}{b}-\frac{\alpha}{2}{c}))=\\=
\frac{1}{4}\Big((3\psi-\beta-2\gamma)ab+\beta{a(ab)}-(\alpha+1)a(bc)+a\big(-a(bc)+3b(ac)+c(ab)\big)\Big).
\end{multline*}
By Lemma~\ref{l:abc}, we have $a(ab)=\frac{1}{2}(\alpha{a}+ab)$, $a(a(bc))=\frac{1}{2}(\psi{a}+a(bc))$.
Since $b(ac)+c(ab)-\frac{1}{2}bc-\frac{\psi}{2}a-\frac{\gamma}{2}b-\frac{\alpha}{2}c\in A_0(a)$ and $b(ac)-c(ab)\in A_{1/2}(a)$,
we have $a(b(ac)+c(ab))=a(\frac{\psi}{2}a+\frac{1}{2}bc+\frac{\gamma}{2}b+\frac{\alpha}{2}c)=\frac{1}{2}(\psi{a}+a(bc)+\gamma{ab}+\alpha{ac})$
and $a(b(ac)-c(ab))=\frac{1}{2}(b(ac)-c(ab)).$ Then 
\begin{multline*}
(3b(ac)+c(ab))=a(2(b(ac)+c(ab))+(b(ac)-c(ab)))=\\=\psi{a}+\gamma{ab}+\alpha{ac}+a(bc)+\frac{1}{2}b(ac)-\frac{1}{2}c(ab).
\end{multline*}

Therefore, 
\begin{multline*}
a(b(x+y-\frac{1}{2}{bc}-\frac{\gamma}{2}{b}-\frac{\alpha}{2}{c}))=
\frac{1}{4}\big((3\psi-\beta-2\gamma)ab+\frac{\beta}{2}(\alpha{a}+ab)+(-\alpha-1)a(bc)-\\-\frac{1}{2}(\psi{a}+a(bc))+\psi{a}+\gamma{ab}+\alpha{ac}+a(bc)+\frac{1}{2}b(ac)-\frac{1}{2}c(ab)\big)=\\=\frac{1}{8}\big((\alpha\beta+\psi)a+(6\psi-\beta-2\gamma)ab+2\alpha{ac}+(-2\alpha-1)a(bc)+b(ac)-c(ab)\big).
\end{multline*}
It remains to add $\frac{(ab)}{2}(bc+\gamma{b}+\alpha{c})-(ab)x$ to the latter expression.
From the above calculations, we see that 
$$\frac{(ab)}{2}(bc+\gamma{b}+\alpha{c})=
\frac{1}{8}\big((2\alpha\gamma+\psi)b+ (\beta+2\gamma)ab+\alpha{bc}+ a(bc)-b(ac)+(4\alpha+1)c(ab)\big).$$
Summing up all expressions, we find that
\[
(ab)y=\frac{1}{8}\Big(\alpha\beta{a}+\alpha\gamma{b}+4\psi{ab}+\alpha{bc}+\alpha{ac}-2\alpha{a(bc)}-2\alpha{b(ac)}+4\alpha{c(ab)}\Big).\]

Finally, observe that the remaining case $x=(bc)a$ is similar to $x=(ac)b$, since the cases for $ab$ and $ba$ are symmetrical.
\end{proof}

\begin{lem}\label{l:abc-bac-cab} If $x\in{B}$ then $((ab)c)x$, $((bc)a)x$, and $((ac)b)x$ belong in $B$. Moreover, if $x\in\mathcal{B}$ and $\eta=1/2$ then these products are as in Table~\ref{t:prod}.
\end{lem}
\begin{proof} Clearly, we can assume that $x\in\mathcal{B}$ and verify only the products $((bc)a)x$.
Observe that the cases $x\in\{a,b,c,ab,ac,bc\}$ are considered in Lemmas~\ref{l:abc} and \ref{l:ab-ac-bc}.

Suppose that $x=(cb)a$. Then $((bc)a)x=((bc)a-\eta{bc})((cb)a)+\eta(bc)((cb)a)$. By Lemma~\ref{l:A0-Aeta}, we have $(bc)a-\eta{bc}\in A_0(a)+A_1(a)$.
It follows from Lemma~\ref{l:seress} that $((cb)a-\eta{cb})((cb)a)=(((cb)a-\eta{bc})(bc))a$. Now Lemmas~\ref{l:abc} and \ref{l:ab-ac-bc} imply that $(((bc)a-\eta{bc})(bc))a\in B$
and $\eta{bc}((bc)a)\in B$, so $((bc)a)x\in B$. Now we find $(a(bc))^2$ for $\eta=1/2$.
We know that  $(a(bc))^2=a((bc)(a(bc)-\frac{1}{2}{bc}))+\frac{1}{2}bc(a(bc)).$
By Lemma~\ref{l:ab-ac-bc}, we have 
\begin{multline*}
(bc)(a(bc)-\frac{1}{2}{bc})=\\=\frac{1}{8}((\beta\gamma-\beta)b+(\alpha\beta-\beta)c+\beta{ab} +(4\psi-2\beta)bc+\beta{ac}+4\beta{a(bc)}-2\beta{b(ac)}-2\beta{c(ab})).
\end{multline*}
Then Lemma~\ref{l:abc} implies that 
$$a((bc)(a(bc)-\frac{1}{2}{bc}))=\frac{1}{16}((\alpha\beta+\beta\gamma+2\beta\psi)a-\beta{ab}-\beta{ac}+ (8\psi-2\beta)a(bc)).$$
Finally, by Lemma~\ref{l:ab-ac-bc}, we know that 
$$\frac{1}{2}bc(a(bc))=\frac{1}{16}(\beta\gamma{b}+\alpha\beta{c}+\beta{ab}+4\psi{bc}+\beta{ac}+4\beta{a(bc)} -2\beta{b(ac)}-2\beta{c(ab)})$$
and hence
$$(a(bc))^2=\frac{1}{16}((\alpha\beta+\beta\gamma+2\beta\psi)a+\beta\gamma{b}+\alpha\beta{c}+4\psi{bc} +(2\beta+8\psi){a(bc)}-2\beta{b(ac)}-2\beta{c(ab)}).$$

Suppose that $x=(ab)c$, $y=(ac)b$ and $z=(bc)a$. 
Then $(a(bc))(x+y)=(a(bc))(x+y-\eta{bc}-\eta\gamma{b}-\eta\alpha{c})+(a(bc))(\eta{bc}+\eta\gamma{b}+\eta\alpha{c})$.
By Lemma~\ref{l:A0-Aeta}, $x+y-\eta{bc}-\eta\gamma{b}-\eta\alpha{c}\in A_0(a)+A_1(a)$. Therefore
$(a(bc))(x+y-\eta{bc}-\eta\gamma{b}-\eta\alpha{c})=a(bc(x+y-\eta{bc}-\eta\gamma{b}-\eta\alpha{c}))$.
Now Lemmas~\ref{l:abc} and \ref{l:ab-ac-bc} imply that $a(bc(x+y-\eta{bc}-\eta\gamma{b}-\eta\alpha{c}))\in B$
and $(a(bc))(\eta{bc}+\eta\gamma{b}+\eta\alpha{c})\in B$. So $z(x+y)=(a(bc))(x+y)\in B$.
Similarly, we get that $y(x+z)$ and $x(z+y)$ belong in $B$. Then $yz-xz=y(x+z)-x(z+y)\in B$ and hence $yz\pm zx\in B$.
Therefore, we infer that both products $zx$ and $zy$ belong in $B$. Now we find the expression of $zy$ for $\eta=1/2$.
We know that $z(x+y)=a(bc(x+y-\frac{1}{2}bc-\frac{\gamma}{2}b-\frac{\alpha}{2}c))+z(\frac{1}{2}bc+\frac{\gamma}{2}b+\frac{\alpha}{2}c)$. By Lemmas~\ref{l:abc} and \ref{l:ab-ac-bc}, we have
\begin{multline*}
bc(x+y-\frac{1}{2}bc-\frac{\gamma}{2}b-\frac{\alpha}{2}c)=\\=\frac{1}{8}
((2\psi-\beta\gamma-\beta)b+(2\psi-\alpha\beta-\beta)c+\beta{ab}+(4\psi-2\alpha-2\beta-2\gamma)bc+\beta{ac}+
2\beta{b(ac)}+2\beta{c(ab)}).
\end{multline*}
Then 
\begin{multline*}
a(bc(x+y-\frac{1}{2}bc-\frac{\gamma}{2}b-\frac{\alpha}{2}c))=\\=\frac{1}{16}
\big((\alpha\beta+\beta\gamma+2\beta\psi)a+(4\psi-\beta)ab+(4\psi-\beta)ac+(8\psi-4\alpha-2\beta-4\gamma)a(bc)\big).
\end{multline*}
Moreover, using Lemmas~\ref{l:abc} and \ref{l:ab-ac-bc}, we find that
\begin{multline*}
z(\frac{1}{2}bc+\frac{\gamma}{2}b+\frac{\alpha}{2}c)=\frac{1}{16}
\big(\gamma(\beta+2\psi)b+\alpha(\beta+2\psi)c+\beta(2\gamma+1)ab+4(\alpha\gamma+\psi)bc+ \beta(2\alpha+1)ac+\\+2(\alpha+2\beta+\gamma)a(bc)+2(\gamma-\alpha-\beta)b(ac)+2(\alpha-\beta-\gamma)c(ab)
\big).
\end{multline*}
Finally, we get that
\begin{multline*}
z(x+y)=\frac{1}{16}
(\beta(\alpha+\gamma+2\psi)a+\gamma(\beta+2\psi)b+\alpha(\beta+2\psi)c+2(\beta\gamma+2\psi)ab+ 4(\alpha\gamma+\psi)bc+\\+2(\alpha\beta+2\psi)ac+2(4\psi-\alpha+\beta-\gamma)a(bc)+2(\gamma-\alpha-\beta)b(ac) +2(\alpha-\beta-\gamma)c(ab)).
\end{multline*}
In a similar manner, we find that
\begin{multline*}
x(z+y)=\frac{1}{16}
(\beta(\alpha+2\psi)a+\gamma(\alpha+2\psi)b+\alpha(\beta+\gamma+2\psi)c+4(\beta\gamma+\psi)ab+ 2(\alpha\gamma+2\psi)bc+\\+2(\alpha\beta+2\psi)ac+2(\beta-\alpha-\gamma)a(bc)+2(\gamma-\alpha-\beta)b(ac)+2(4\psi+\alpha-\beta-\gamma)c(ab)),
\end{multline*}
\begin{multline*}
y(x+z)=\frac{1}{16}
(\beta(\gamma+2\psi)a+\gamma(\alpha+\beta+2\psi)b+\alpha(\gamma+2\psi)c+2(\beta\gamma+2\psi)ab+ 2(\alpha\gamma+2\psi)bc+\\+4(\alpha\beta+\psi)ac+2(\beta-\alpha-\gamma)a(bc)+2(4\psi-\alpha-\beta+\gamma)b(ac)+2(\alpha-\beta-\gamma)c(ab)).
\end{multline*}
Using these expression, we see that $$yz-xz=y(x+z)-x(y+z)=\frac{1}{16}(\beta(\gamma-\alpha)a+\beta\gamma{b}-\alpha\beta{c}-2\beta\gamma{ab}+2\alpha\beta{ac}+8\psi{b(ac)}-8\psi(c(ab)).$$
Since $2yz=z(x+y)+z(y-x)$, we have
\begin{multline*}
a(bc)\cdot{b(ac)}=\frac{1}{16}(\beta(\gamma+\psi)a+\gamma(\beta+\psi)b+\alpha\psi{c}+2\psi{ab}+ 2(\alpha\gamma+\psi)bc+2(\alpha\beta+\psi)ac+\\+(4\psi-\alpha+\beta-\gamma)a(bc)+(4\psi-\alpha-\beta+\gamma)b(ac) +(\alpha-\beta-\gamma-4\psi)c(ab)).
\end{multline*}
\end{proof}

\begin{table}
\begin{center}
\begingroup
\setlength{\tabcolsep}{20pt} %
\renewcommand{\arraystretch}{1.1}
{\small
$\begin{tabu}[h!]{|c|}
\hline
a\cdot a = a, b\cdot b=b, c\cdot c=c \\ \hline
a\cdot b = ab, a\cdot c=ac, b\cdot c=bc \\ \hline
a\cdot bc = a(bc), b\cdot ac = b(ac), c\cdot ab = c(ab)  \\ \hline
\begin{tabu}{@{}c@{}}
a\cdot ab = \frac{\alpha}{2}a+\frac{1}{2}{ab}, a\cdot ac = \frac{\gamma}{2}a+\frac{1}{2}{ac}, b\cdot ab = \frac{\alpha}{2}b+\frac{1}{2}{ab}, \\ 
b\cdot bc = \frac{\beta}{2}b+\frac{1}{2}{bc}, c\cdot ac =\frac{\gamma}{2}c+\frac{1}{2}{ac}, c\cdot bc = \frac{\beta}{2}c+\frac{1}{2}{bc}
\end{tabu}\\ \hline
a\cdot a(bc) =  \frac{\psi}{2}a+\frac{1}{2}{a(bc)}, b\cdot b(ac) =  \frac{\psi}{2}b+\frac{1}{2}{b(ac)}, c\cdot c(ab) =  \frac{\psi}{2}c+\frac{1}{2}{c(ab)}  \\ \hline
\begin{tabu}{@{}c@{}}
a\cdot b(ac) =  \frac{1}{4}(\psi{a}+\gamma{ab}+\alpha{ac}+a(bc)+b(ac)-c(ab)) \\
a\cdot c(ab) =  \frac{1}{4}(\psi{a}+\gamma{ab}+\alpha{ac}+a(bc)-b(ac)+c(ab)) \\
b\cdot a(bc) =  \frac{1}{4}(\psi{b}+\beta{ab}+\alpha{bc}+a(bc)+b(ac)-c(ab)) \\
b\cdot c(ab) =  \frac{1}{4}(\psi{b}+\beta{ab}+\alpha{bc}-a(bc)+b(ac)+c(ab)) \\
c\cdot a(bc) =  \frac{1}{4}(\psi{c}+\gamma{bc}+\beta{ac}+a(bc)-b(ac)+c(ab)) \\
c\cdot b(ac) =  \frac{1}{4}(\psi{c}+\gamma{bc}+\beta{ac}-a(bc)+b(ac)+c(ab))
\end{tabu}
 \\ \hline
\begin{tabu}{@{}c@{}} 
ab\cdot ab = \frac{\alpha}{4}(a+b)+\frac{\alpha}{2}ab, bc\cdot bc = \frac{\beta}{4}(b+c)+\frac{\beta}{2}bc , ac\cdot ac = \frac{\gamma}{4}(a+c)+\frac{\gamma}{2}ac
\end{tabu}\\ \hline
\begin{tabu}{@{}c@{}} 
ab\cdot bc = \frac{1}{4}(\psi{b}+\beta{ab}+\alpha{bc}+a(bc)-b(ac)+c(ab))  \\
bc\cdot ac = \frac{1}{4}(\psi{c}+\gamma{bc}+\beta{ac}+a(bc)+b(ac)-c(ab))  \\
ab\cdot ac = \frac{1}{4}(\psi{a}+\gamma{ab}+\alpha{ac}-a(bc)+b(ac)+c(ab))  
\end{tabu}
\\ \hline
\begin{tabu}{@{}c@{}}
 \begin{tabu}{@{}c@{}} ab\cdot a(bc) = \frac{1}{8}((\alpha\beta+\psi)a+\psi{b}+2\psi{ab}+\alpha{bc}+2\alpha{a(bc)}) \end{tabu}  \\
 \begin{tabu}{@{}c@{}} ab\cdot b(ac) = \frac{1}{8}(\psi{a}+(\alpha\gamma+\psi)b+2\psi{ab}+\alpha{ac}+2\alpha{b(ac)}) \end{tabu} \\
 \begin{tabu}{@{}c@{}} bc\cdot b(ac) = \frac{1}{8}((\beta\gamma+\psi)b+\psi{c}+2\psi{bc}+\beta{ac}+2\beta{b(ac)}) \end{tabu} \\
 \begin{tabu}{@{}c@{}} bc\cdot c(ab) = \frac{1}{8}(\psi{b}+(\alpha\beta+\psi)c+\beta{ab}+2\psi{bc}+2\beta{c(ab)}) \end{tabu} \\
 \begin{tabu}{@{}c@{}} ac\cdot a(bc) = \frac{1}{8}((\beta\gamma+\psi){a}+\psi{c}+\gamma{bc}+2\psi{ac}+2\gamma{a(bc)}) \end{tabu}  \\
 \begin{tabu}{@{}c@{}} ac\cdot c(ab) = \frac{1}{8}(\psi{a}+(\alpha\gamma+\psi)c+\gamma{ab}+2\psi{ac}+2\gamma{c(ab)}) \end{tabu} 
\end{tabu} \\ \hline
ab\cdot c(ab) =\frac{1}{8}(\alpha\beta{a}+\alpha\gamma{b}+4\psi{ab}+\alpha{bc}+\alpha{ac}-2\alpha{a(bc)}-2\alpha{b(ac)}+4\alpha{c(ab)}) \\
bc\cdot a(bc) =\frac{1}{8}(\beta\gamma{b}+\alpha\beta{c}+\beta{ab}+4\psi{bc}+\beta{ac}+4\beta{a(bc)}-2\beta{b(ac)}-2\beta{c(ab)}) \\
ac\cdot b(ac) =\frac{1}{8}(\beta\gamma{a}+\alpha\gamma{c}+\gamma{ab}+\gamma{bc}+4\psi{ac}-2\gamma{a(bc)}+4\gamma{b(ac)}-2\gamma{c(ab)}) \\ \hline
\begin{tabu}{@{}c@{}}
(a(bc))^2=\frac{1}{16}(\beta(\alpha+\gamma+2\psi)a+\beta\gamma{b}+\alpha\beta{c}+4\psi{bc}+(2\beta+8\psi)a(bc)-2\beta{b(ac)}-2\beta{c(ab)}) \\
(b(ac))^2=\frac{1}{16}(\beta\gamma{a}+\gamma(\alpha+\beta+2\psi)b+\alpha\gamma{c}+4\psi{ac} -2\gamma{a(bc)}+(2\gamma+8\psi)b(ac)-2\gamma{c(ab}) \\
(c(ab))^2=\frac{1}{16}(\alpha\beta{a}+\alpha\gamma{b}+\alpha(\beta+\gamma+2\psi)c+ 4\psi{ab}-2\alpha{a(bc)}-2\alpha{b(ac)}+(2\alpha+8\psi)c(ab))

\end{tabu} \\ \hline
a(bc)\cdot{b(ac)}=\frac{1}{16}(\beta(\gamma+\psi)a+\gamma(\beta+\psi)b+\alpha\psi{c}+2\psi{ab}+ 2(\alpha\gamma+\psi)bc+2(\alpha\beta+\psi)ac+\\+(4\psi-\alpha+\beta-\gamma)a(bc)+(4\psi-\alpha-\beta+\gamma)b(ac) +(\alpha-\beta-\gamma-4\psi)c(ab)) \\
b(ac)\cdot{c(ab)}=\frac{1}{16}(
\beta\psi{a}+\gamma(\alpha+\psi)b+\alpha(\gamma+\psi)c+2(\beta\gamma+\psi)ab+2\psi{bc}+2(\alpha\beta+\psi){ac} +\\+(\beta-\alpha-\gamma-4\psi)a(bc)+(4\psi-\alpha-\beta+\gamma)b(ac)+(4\psi+\alpha-\beta-\gamma)c(ab))\\
c(ab)\cdot{a(bc)}=\frac{1}{16}(
\beta(\alpha+\psi)a+\gamma\psi{b}+\alpha(\beta+\psi)c+2(\beta\gamma+\psi)ab+2(\alpha\gamma+\psi)bc+2\psi{ac} +\\+(4\psi-\alpha+\beta-\gamma)a(bc)+(\gamma-\alpha-\beta-4\psi)b(ac)+(4\psi+\alpha-\beta-\gamma)c(ab))
\\ \hline
\end{tabu}$}

\caption{Products for $\eta=1/2$}\label{t:prod}
\endgroup
\end{center}
\end{table}

\section{3-generated algebras of Jordan type half}\label{sec:4}
This section is devoted to the properties of the algebras with the table of products as in
Table~\ref{t:prod} as well as the proof of Theorem~\ref{thm:2}. Fix a field $\mathbb{F}$ of characteristic not two
and select arbitrary values of the parameters $\alpha,\beta,\gamma,\psi\in\mathbb{F}$.
Assume that $e_1,e_2,\ldots,e_9$ is a basis of a 9-dimensional vector space $V$ over $\mathbb{F}$. Define products on these elements as in Table~\ref{t:prod} where we identify elements of the basis with $a,b,c,ab,bc,ac,a(cb),b(ac),c(ab)$ in this table. The space $V$ with the thus-defined product forms an $\mathbb{F}$-algebra which we denote by $A(\alpha,\beta,\gamma,\psi)$.
Now we prove some properties of this algebra. Note that we use GAP extensive throughout this section to do calculations in the algebra. The code we used can be found on GitHub \cite{Git}.

\subsection{Axial properties of $A(\alpha,\beta,\gamma,\psi)$}
In this subsection we prove that $A(\alpha,\beta,\gamma,\psi)$ is a primitive axial algebra of Jordan type half as well as a Jordan algebra. In some sense, this is a converse of Theorem~\ref{thm:1}.

\begin{prop}\label{p:table}  The algebra $A(\alpha,\beta,\gamma,\psi)$ is an axial algebra of Jordan type half generated by the three primitive idempotents $e_1$, $e_2$, and $e_3$.	
\end{prop}
\begin{proof} For convenience, set $A=A(\alpha,\beta,\gamma,\psi)$.
By definition, we see that $e_1^2=e_1$,
	$e_2^2=e_2$, and $e_3^2=e_3$. Moreover, since $e_4=e_1e_2$, $e_5=e_2e_3$, $e_6=e_1e_2$, $e_7=e_1(e_2e_3)$, $e_8=e_2(e_1e_3)$, and $e_9=e_3(e_1e_2)$, we infer that $A$ is generated as algebra by three idempotents 
	$e_1$, $e_2$, and $e_3$. Denote by $A_\lambda(e_i)$ the eigenspace of $ad_{e_i}$ associated with $\lambda\in\mathbb{F}$.
	As in Theorem~\ref{thm:1}, we show that $A_0(e_1)\subseteq\langle e_4-\frac{1}{2}{e_2}-\frac{\alpha}{2}e_1, e_6-\frac{1}{2}{e_3}-\frac{\gamma}{2}e_1, e_7-\frac{1}{2}{e_5}-\frac{\psi}{2}{e_1}, 	e_8+e_9-\frac{1}{2}{e_5}-\frac{\alpha}{2}{e_3}-\frac{\gamma}{2}{e_2}-\frac{\psi}{2}e_1\rangle$.

	Since $e_1^2=e_1$ and $e_1e_2=e_4$, we have
	 $e_1(e_4-\frac{\alpha}{2}e_1-\frac{1}{2}{e_2})=e_1e_4-\frac{\alpha}{2}e_1-\frac{1}{2}{e_4}=0$. Similarly, we obtain 
	$e_1(e_5-\frac{\gamma}{2}e_1-\frac{1}{2}{e_3})=0$. Since $e_1\cdot e_7=\frac{\psi}{2}e_1+\frac{1}{2}e_7$ and $e_1\cdot e_5=e_7$,
	we infer that $e_1\cdot(e_7-\frac{1}{2}e_5-\frac{\psi}{2}e_1)=0$. It remains to prove that $e_1\cdot(e_8+e_9-\frac{1}{2}{e_5}-\frac{\alpha}{2}{e_3}-\frac{\gamma}{2}{e_2}-\frac{\psi}{2}e_1)=0$. It follows from the table of products that $e_1(e_8+e_9)=\frac{\psi}{2}e_1+\frac{\gamma}{2}e_4+\frac{\alpha}{2}e_6+\frac{1}{2}e_7$.
	Moreover, $e_1\cdot e_5=e_7$, $e_1\cdot e_3=e_6$, $e_1\cdot e_2=e_4$.
	So $e_1\cdot(e_8+e_9-\frac{1}{2}{e_5}-\frac{\alpha}{2}{e_3}-\frac{\gamma}{2}{e_2}-\frac{\psi}{2}e_1)=\frac{\psi}{2}e_1+\frac{\gamma}{2}e_4+\frac{\alpha}{2}e_6+\frac{1}{2}e_7-\frac{1}{2}e_7-\frac{\alpha}{2}e_6-\frac{\alpha}{2}e_4-\frac{\psi}{2}e_1=0$.
	
	According to Table~\ref{t:prod}, we see that $e_1(e_8-e_9)=\frac{1}{2}(e_8-e_9)$, 
	$e_1(e_4-\alpha{e}_1)=\frac{1}{2}(e_4-\alpha{e}_1)$, $e_1(e_6-\gamma{e}_1)=\frac{1}{2}(e_6-\gamma{e}_1)$,
	$e_1(e_7-\psi{e}_1)=\frac{1}{2}(e_7-\psi{e}_1)$, so we find four elements in $A_{1/2}(e_1)$. Therefore, we have nine eigenvectors in $A_1(e_1)$, $A_0(e_1)$, and $A_{1/2}(e_1)$. Using GAP, we see that the matrix comprising of the corresponding coordinate vectors has determinant $\pm\frac{1}{4}$, depending on the order of these vectors. This implies that these nine vectors are linearly independent and hence $A=A_1(e_1)\oplus A_0(e_1)\oplus A_1(e_1)$. Moreover,  it is true that
	$$A_1(e_1)=\langle e_1 \rangle, A_{1/2}(e_1)=\langle e_8-e_9, e_4-\alpha{e}_1, e_6-\gamma{e}_1, e_7-\psi{e}_1\rangle,$$ 
	$$A_0(e_1)=\langle e_4-\frac{\alpha}{2}e_1-\frac{1}{2}{e_2}, e_6-\frac{\gamma}{2}e_1-\frac{1}{2}{e_3}, 
	e_7-\frac{\psi}{2}{e_1}-\frac{1}{2}{e_5}, e_8+e_9-\frac{1}{2}{e_5}-\frac{\alpha}{2}{e_3}-\frac{\gamma}{2}{e_2}-\frac{\psi}{2}e_1\rangle.$$
	
	It remains to see that the eigenspaces satisfy the fusion rules.
	It is clear that 
	$$A_1(e_1)\cdot A_0(e_1)=\{0\}, A_1(e_1)\cdot A_1(e_1)=A_1(e_1), \text{ and }A_1(e_1)\cdot A_{1/2}(e_1)=A_{1/2}(e_1).$$
	We need to verify the following inclusions: 
	$$A_0(e_1)\cdot A_0(e_1)\subseteq A_0(e_1), A_0(e_1)\cdot A_{1/2}(e_1)\subseteq A_{1/2}(e_1), A_{1/2}(e_1)\cdot A_{1/2}(e_1)\subseteq A_0(e_1)\oplus A_1(e_1).$$
	Denote $f_1=e_4-\frac{\alpha}{2}e_1-\frac{1}{2}{e_2}$, $f_2=e_6-\frac{\gamma}{2}e_1-\frac{1}{2}{e_3}$, $f_3=e_7-\frac{\psi}{2}{e_1}-\frac{1}{2}{e_5}$, and	$f_4=e_8+e_9-\frac{1}{2}{e_5}-\frac{\alpha}{2}{e_3}-\frac{\gamma}{2}{e_2}-\frac{\psi}{2}e_1$. Thus $f_1$, $f_2$, $f_3$, and $f_4$ form a basis of $A_0(e_1)$.
	Similarly, denote $g_1=e_8-e_9$, $g_2=e_4-\alpha{e}_1$, $g_3=e_6-\gamma{e}_1$, and $g_4=e_7-\psi{e}_1$: in particular elements $g_1$, $g_2$, $g_3$, and $g_4$ form a basis of $A_{1/2}(e_1)$.
	Now we find the products $f_i\cdot f_j$, $f_i\cdot g_j$, and $g_i\cdot g_j$, where $i, j\in\{1,2,3,4\}$, using Table~\ref{t:prod}.
	We put information in Tables~\ref{t:prod_A0}-\ref{t:prod_A012}, there we show that for all $i$, $j\in\{1,2,3,4\}$ it is true that $f_i\cdot f_j$ lies in $A_0(e_1)$, $g_i\cdot g_j$ lies in $A_1(e_1)\oplus A_0(e_1)$,
	and $f_i\cdot g_j$ lies in 	$A_{1/2}(e_1)$ by expressing the products in the corresponding bases. All calculation can be verified with the aid of GAP.
	By the symmetry of the generating idempotents, we conclude that $e_2$ and $e_3$ are also primitive axes of $A$. This completes the proof of the proposition.

\begin{table}
\begin{center}
\begingroup
\setlength{\tabcolsep}{20pt} %
\renewcommand{\arraystretch}{1.3}

{\tiny$\begin{tabu}[h!]{|c||c|c|c|c|}
\hline
\ast & f_1 & f_2 & f_3 & f_4 \\ \hline\hline
f_1 & \frac{\alpha-1}{2}f_1  & \frac{1}{4}(\gamma f_1+\alpha f_2 -f_3-f_4) & \frac{1}{4}( (\psi-\beta)f_1+(\alpha-1)f_3) & \frac{1}{4}((3\psi-\beta-\gamma)f_1+\alpha f_2-\alpha f_3 -f_4)  \\ \hline
f_2 &  & \frac{\gamma-1}{2}f_2 & \frac{1}{4}((\psi-\beta)f_2+(\gamma-1)f_3) & \frac{1}{4}(\gamma f_1+(3\psi-\alpha-\beta)f_2-\gamma f_3-f_4) \\ \hline
f_3 &  & & \frac{1}{8}(\beta(f_4-f_1-f_2)+(4\psi-3\beta)f_3) & \frac{1}{8}( (2\psi-\beta)(f_1+f_2)+(4\psi-2\alpha-\beta-2\gamma)f_3-\beta f_4)  \\ \hline
f_4 &  &  &  & \frac{1}{8}((4\psi-\beta)(f_1+f_2-f_3)+(8\psi-4\alpha-3\beta-4\gamma)f_4) \\ \hline
\end{tabu}$}

\caption{Products for $A_0(e_1)$}\label{t:prod_A0}
\endgroup
\end{center}
\end{table}
	
\begin{table}
\begin{center}
\begingroup
\renewcommand{\arraystretch}{1.3}

{\tiny$\begin{tabu}[h!]{|c||c|c|c|c|}
\hline
\ast & g_1 & g_2 & g_3 & g_4 \\ \hline\hline
g_1 & \begin{tabu}{@{}c@{}} \frac{1}{16}((\alpha\beta+\beta\gamma-4\alpha\beta\gamma-2\beta\psi+4\psi^2)e_1-\\-4\beta\gamma f_1-4\alpha\beta f_2+(8\psi-2\beta)f_3+2\beta f_4) \end{tabu} & \begin{tabu}{@{}c@{}} \frac{1}{8}((\psi-\alpha\beta)e_1-\\-2\psi f_1+2\alpha f_3) \end{tabu} & \begin{tabu}{@{}c@{}}  \frac{1}{8}((\beta\gamma-\psi)e_1+\\+2\psi f_2-2\gamma f_3) \end{tabu} &   \begin{tabu}{@{}c@{}} \frac{1}{16}( (\beta\gamma-\alpha\beta)e_1-\\-2\beta\gamma f_1+2\alpha\beta f_2) \end{tabu} \\ \hline
 g_2 &  & \frac{1}{4}((\alpha-\alpha^2)e_1-2\alpha f_1) &  \begin{tabu}{@{}c@{}}   \frac{1}{4}((\psi-\alpha\gamma)e_1-\\-\gamma f_1-\alpha f_2-f_3+f_4) \end{tabu} &   \begin{tabu}{@{}c@{}} \frac{1}{8}((\alpha\beta-2\alpha\psi+\psi)e_1-\\-2\psi f_1-2\alpha f_3) \end{tabu}  \\ \hline
g_3 &  &  & \frac{1}{4}((\gamma-\gamma^2)e_1-2\gamma f_2)  &  \begin{tabu}{@{}c@{}}  \frac{1}{8}((\beta\gamma-2\gamma\psi+\psi)e_1-\\-2\psi f_2-2\gamma f_3) \end{tabu} \\ \hline
g_4 &  &  &  &  \begin{tabu}{@{}c@{}}  \frac{1}{16}((\alpha\beta+\beta\gamma+2\beta\psi-4\psi^2)e_1+ \\ +(2\beta-8\psi)f_3-2\beta f_4) \end{tabu} \\ \hline
\end{tabu}$}

\caption{Products for $A_{1/2}(e_1)$}\label{t:prod_A12}
\endgroup
\end{center}
\end{table}

\begin{table}
\begin{center}
\begingroup
\setlength{\tabcolsep}{10pt} %
\renewcommand{\arraystretch}{1.3}

{\tiny$\begin{tabu}[h!]{|c||c|c|c|c|}
\hline
\ast & g_1 & g_2 & g_3 & g_4 \\ \hline\hline
f_1 & \begin{tabu}{@{}c@{}} \frac{1}{8}((2\alpha-1)g_1+(\beta-2\psi)g_2+\\+(2\alpha-1)g_4) \end{tabu} & \frac{\alpha-1}{4}g_2 & \frac{1}{4}(-g_1+\gamma g_2-g_4) & \frac{1}{8}(-g_1+(2\psi-\beta)g_2-g_4) \\ \hline 
f_2 &  \begin{tabu}{@{}c@{}} \frac{1}{8}( (2\gamma-1)g_1+(2\psi-\beta)g_3+\\+(1-2\gamma)g_4) \end{tabu} & \frac{1}{4}(g_1+\alpha g_3-g_4) & \frac{\gamma-1}{4}g_3 & \frac{1}{8}(g_1+(2\psi-\beta)g_3-g_4) \\ \hline
f_3 &  \begin{tabu}{@{}c@{}} \frac{1}{16}((4\psi-2\beta)g_1+(\beta-2\beta\gamma)g_2+\\+(2\alpha\beta-\beta)g_3) \end{tabu} & \frac{1}{8}(g_1-\beta g_2+(2\alpha-1)g_4) & \frac{1}{8}(-g_1-\beta g_3+(2\gamma-1)g_4) & \frac{1}{16}(-\beta g_2-\beta g_3+(4\psi-2\beta)g_4) \\ \hline
f_4 &  \begin{tabu}{@{}c@{}}  \frac{1}{16}(4\psi-2\beta)g_1+(2\beta\gamma+\beta-4\psi)g_2+\\+(4\psi-2\alpha\beta-\beta)g_3+4(\alpha-\gamma)g_4) \end{tabu} &  \begin{tabu}{@{}c@{}}  \frac{1}{8}(g_1+(4\psi-\beta-2\gamma)g_2+\\+2\alpha g_3-(2\alpha+1)g_4) \end{tabu} &  \begin{tabu}{@{}c@{}}  \frac{1}{8}(-g_1+ 2\gamma g_2-(2\gamma+1)g_4+\\+(4\psi-2\alpha-\beta)g_3 ) \end{tabu} &   \begin{tabu}{@{}c@{}}  \frac{1}{16}((4\psi-\beta)g_2+(4\psi-\beta)g_3+\\+(4\psi-4\alpha-2\beta-4\gamma)g_4) \end{tabu} \\ \hline
\end{tabu}$}

\caption{Products of $A_0(e_1)$ and $A_{1/2}(e_1)$}\label{t:prod_A012}
\endgroup
\end{center}
\end{table}

\end{proof}

\begin{prop} The algebra $A(\alpha,\beta,\gamma,\psi)$ is a Jordan algebra.
\end{prop}
\begin{proof} Suppose that $x=\sum_{i=1}^9x_ie_i$ and $y=\sum_{i=1}^9y_ie_i$ are arbitrary elements of $A(\alpha,\beta,\gamma,\psi)$. Cleary, $xy=yx$ and hence we need to show that $(x^2y)x=x^2(yx)$. We verify this identity in GAP assuming that $x_i$, $y_i$ are new variables and comparing expressions for both parts of the equality expanded with the aid of Table~\ref{t:prod}. The code can be found in~\cite{Git}.
\end{proof}
\begin{rem} It is known that any idempotent of a Jordan algebra gives a so-called Pierce decomposition and eigenspaces of its adjoint operator obey the required fusion rules. So this proposition also yields that an algebra with table of product as in Table~\ref{t:prod} is an axial algebra of Jordan type half. 
\end{rem}

We conclude this subsection with the following statement which is a consequence of Theorem~\ref{thm:1} and Proposition~\ref{p:table}.
\begin{prop} Suppose that $A$ is a primitive axial algebra of Jordan type half generated by three primitive axes $x$, $y$, and $z$ with the Frobenius form as in Lemma~\ref{l:frobenius}. If $(x, y)=\alpha$, $(y, z)=\beta$,
$(x, z)=\gamma$, and $(x, yz)=\psi$ then $A$ is a homomorphic image of $A(\alpha,\beta,\gamma,\psi)$.
\end{prop}

\subsection{Gram matrix}

By Proposition~\ref{p:table}, the algebra $A(\alpha,\beta,\gamma,\psi)$ is a primitive axial algebra generated by primitive axes $e_1$, $e_2$, and $e_3$. For simplicity, we denote $a=e_1$, $b=e_2$, and $c=e_3$. Then the basis of $A(\alpha,\beta,\gamma,\psi)$ is $a$, $b$, $c$, $ab=e_4$, $bc=e_5$, $ac=e_6$, $a(bc)=e_7$, $b(ac)=e_8$, and $c(ab)=e_9$ as in Table~\ref{t:prod}. By Lemma~\ref{l:frobenius}, there is a Frobenius symmetric bilinear form $(\cdot,\cdot)$ on $A(\alpha,\beta,\gamma,\psi)$ such that $(a,a)=(b,b)=(c,c)=1$.
In this subsection we find the Gram matrix and the radical of this form. It follows from Proposition~\ref{p:summary} that
if at least two of scalars $\alpha$, $\beta$, and $\gamma$ are non-zero then the radical includes all proper ideals of the algebra.

\begin{prop}\label{p:find_gram} The Gram matrix of the form for $a$, $b$, $c$, $ab$, $bc$, $ac$, $a(bc)$, $b(ac)$, $c(ab)$ is as in Table~\ref{t:gram} 
\end{prop}
\begin{proof} 
First, we restore the basic values $(a,b)$, $(b,c)$, and $(a,c)$ from the table of products.
Since the form associates with the algebra product, 
we have $(a,b)=(a,ab)=(a,a(ab))=\frac{1}{2}(a, \alpha{a}+ab)=\frac{\alpha}{2}+\frac{1}{2}(a,ab)=\frac{\alpha}{2}+\frac{1}{2}(a,b)$
and hence $(a,b)=\alpha$. Similarly, we find that
$(b,c)=\beta$, $(a,c)=\gamma$ and $(a,bc)=(b,ac)=(c,ab)=\psi$.
Then $(a,ab)=(b,ab)=(a,b)=\alpha$, $(a,ac)=(c,ac)=(a,c)=\gamma$, $(b,bc)=(c,bc)=(b,c)=\beta$.
Moreover, we have $(a,a(bc))=(a,bc)=\psi$ and the same is true for $(b,b(ac))$ and $(c,c(ab))$.
Lemma~\ref{l:cab} implies that  $(a,b(ac))=(ab,ac)=(a,c(ab))=\frac{1}{2}(\alpha\gamma+\psi)$.

Since $a(ab)=\frac{1}{2}(\alpha{a}+ab)$, we have $(ab,ab)=(b,a(ab))=(b,\frac{1}{2}(\alpha{a}+ab)=\frac{\alpha^2}{2}+\frac{\alpha}{2}=\frac{\alpha}{2}(\alpha+1)$. 

Now we see that
\begin{multline*}
(ab,a(bc))=(a(ab),bc)=(\frac{1}{2}(\alpha{a}+ab), bc)=\frac{1}{2}(\alpha{a},bc)+(\frac{1}{2}ab, bc)=\\=\frac{\alpha\psi}{2}+\frac{1}{4}(\alpha\beta+\psi)=\frac{1}{4}(2\alpha\psi+\alpha\beta+\psi).
\end{multline*}
Similarly, we get that $(ab,b(ac))=\frac{1}{4}(2\alpha\psi+\alpha\gamma+\psi)$. Now we find $(ab, c(ab))$. Since $b(c(ab))=\frac{1}{4}(\psi{b}+\beta{ab}+\alpha{bc}-a(bc)+b(ac)+c(ab))$,
we have 
\begin{multline*}
(ab, c(ab))=(a,b(c(ab)))=\\=\frac{1}{4}((a,\psi{b})+(a,\beta{ab})+(a,\alpha{bc})-(a,a(bc))+(a,b(ac))+(a,c(ab)))=\\=
\frac{1}{4}(\alpha\psi+\alpha\beta+\alpha\psi-\psi+\alpha\gamma+\psi)=\frac{\alpha}{4}(2\psi+\beta+\gamma).\end{multline*}
Similarly, we can obtain values of the form on $bc$ and $ac$.

Now we find the value of $(a(bc), a(bc))$. Since $a\cdot a(bc)=\cfrac{\psi}{2}a+\cfrac{1}{2}a(bc)$, we have
$(a(bc), a(bc))=(bc, a(a(bc)))=\cfrac{\psi}{2}(bc,a)+\cfrac{1}{2}(bc, a(bc))=\cfrac{\psi^2}{2}+\cfrac{\beta}{8}(\alpha+\gamma+2\psi)=\cfrac{1}{8}(\alpha\beta+\beta\gamma+2\psi\beta+4\psi^2)$.

Finally, we evaluate the value of $(a(bc), b(ac))$. Since $a\cdot b(ac) =  \frac{1}{4}(\psi{a}+\gamma{ab}+\alpha{ac}+a(bc)+b(ac)-c(ab))$,
we have 
\begin{multline*}
(a(bc), b(ac))=(bc, a(b(ac))) = \cfrac{1}{4}(bc, \psi{a}+\gamma{ab}+\alpha{ac}+a(bc)+b(ac)-c(ab)))=\\=
\cfrac{\psi^2}{4}+\cfrac{\gamma}{8}(\alpha\beta+\psi)+\cfrac{\alpha}{8}(\beta\gamma+\psi)+\cfrac{1}{16}(\beta\alpha+\gamma\beta+2\psi\beta+\beta\gamma+2\beta\psi+\psi-\alpha\beta-2\beta\psi-\psi))=\\=
\cfrac{1}{8}(2\psi^2+2\alpha\beta\gamma+\gamma\psi+\alpha\psi+\beta\gamma+\beta\psi).
\end{multline*}
Similarly, we find values of the form on $b(ac)$ and $c(ab)$.
\end{proof}

\begin{table}
\begin{center}
\begingroup
\setlength{\tabcolsep}{20pt} %
\renewcommand{\arraystretch}{1.5}

{\tiny
$\begin{tabu}[h!]{|c||c|c|c|c|c|c|c|c|c|}
\hline
(,) &  a & b & c & ab & bc & ac & a(bc) & b(ac) &c(ab) \\ \hline\hline
a &  1 & \alpha & \gamma & \alpha & \psi & \gamma & \psi & \frac{1}{2}(\alpha\gamma+\psi) & \frac{1}{2}(\alpha\gamma+\psi) \\ \hline
b &  & 1 & \beta & \alpha & \beta & \psi & \frac{1}{2}(\alpha\beta+\psi) & \psi &  \frac{1}{2}(\alpha\beta+\psi)  \\ \hline 
c &  &  & 1 & \psi & \beta & \gamma & \frac{1}{2}(\beta\gamma+\psi) & \frac{1}{2}(\beta\gamma+\psi) & \psi \\ \hline 
ab &  &  & & \frac{\alpha}{2}(\alpha+1) & \frac{1}{2}(\alpha\beta+\psi) & \frac{1}{2}(\alpha\gamma+\psi) & \frac{1}{4}(\alpha\beta+2\alpha\psi+\psi) & \frac{1}{4}(\alpha\gamma+2\alpha\psi+\psi) & \frac{\alpha}{4}(\beta+\gamma+2\psi) \\ \hline
bc &  & &  &  & \frac{\beta}{2}(\beta+1) & \frac{1}{2}(\beta\gamma+\psi) & \frac{\beta}{4}(\alpha+\gamma+2\psi) & \frac{1}{4}(\beta\gamma+2\beta\psi+\psi) & \frac{1}{4}(\alpha\beta+2\beta\psi+\psi) \\ \hline
ac &  &  & &  &  & \frac{\gamma}{2}(\gamma+1) & \frac{1}{4}(\beta\gamma+2\gamma\psi+\psi) & \frac{\gamma}{4}(\alpha+\beta+2\psi) & \frac{1}{4}(\alpha\gamma+2\gamma\psi+\psi) \\ \hline
a(bc) & &  &  &  &  &  & \begin{tabu}{@{}c@{}} \frac{1}{8}(\alpha\beta+\beta\gamma+\\+2\beta\psi+4\psi^2) \end{tabu} & \begin{tabu}{@{}c@{}c@{}}\frac{1}{8}(2\alpha\beta\gamma+\alpha\psi+\\+\beta\gamma+\beta\psi+\\+\gamma\psi+2\psi^2) \end{tabu} & \begin{tabu}{@{}c@{}c@{}}\frac{1}{8}(2\alpha\beta\gamma+\alpha\beta+\\+\alpha\psi+\beta\psi+\\+\gamma\psi+2\psi^2) \end{tabu} \\ \hline
b(ac) &  &  &  &  & &  &  & \begin{tabu}{@{}c@{}}\frac{1}{8}(\alpha\gamma+\beta\gamma+\\+2\gamma\psi+4\psi^2)\end{tabu} & \begin{tabu}{@{}c@{}c@{}}\frac{1}{8}(2\alpha\beta\gamma+\alpha\gamma+\\+\alpha\psi+\beta\psi+\\+\gamma\psi+2\psi^2) \end{tabu} \\ \hline
c(ab) & & & & & & & & & \begin{tabu}{@{}c@{}}\frac{1}{8}(\alpha\beta+\alpha\gamma+\\+2\alpha\psi+4\psi^2) \end{tabu} \\
 \hline
\end{tabu}$}

\caption{Gram matrix}\label{t:gram}
\endgroup
\end{center}
\end{table}

\begin{lem}\label{l:identity} Consider $e=(\beta-1)a+(\gamma-1)b+(\alpha-1)c+2(ab+bc+ac-a(bc)-b(ac)-c(ab))$
and an element $x\in A(\alpha,\beta,\gamma,\psi)$. Then $ex=(\alpha+\beta+\gamma-2\psi-1)x$. In particular, if 
$(\alpha+\beta+\gamma-2\psi-1)\neq 0$ then $e/(\alpha+\beta+\gamma-2\psi-1)$ is the identity element of the algebra, 
and if $(\alpha+\beta+\gamma-2\psi-1)=0$ then $e\cdot A(\alpha,\beta,\gamma,\psi)=0$. 
\end{lem}
\begin{proof} It suffices to verify the equality  $ex=(\alpha+\beta+\gamma-2\psi-1)x$
for every $x\in\{a,b,c,ab,bc,ac,a(bc),b(ac),c(ab)\}$. Denote $\lambda=\alpha+\beta+\gamma-2\psi-1$.

Suppose that $x=a$. Since $b(ac)+c(ab)-\frac{1}{2}{bc}-\frac{\alpha}{2}{c}-\frac{\gamma}{2}{b}-\frac{\psi}
{2}a\in A_0(a)$, we have $a(b(ac)+c(ab))=\frac{1}{2}{a(bc)}+\frac{\alpha}{2}{ac}+\frac{\gamma}{2}{ab}+\frac{\psi}
{2}a$. Moreover, $a(ab)=\frac{1}{2}(\alpha{a}+ab)$, $a(ac)=\frac{1}{2}(\gamma{a}+ac)$ and $a(a(bc))=\frac{1}{2}(\psi{a}+a(bc))$, So $ae=(\beta-1)a+(\gamma-1)ab+(\alpha-1)ac+\alpha{a}+ab+2a(bc)+\gamma{a}+ac-
\psi{a}-a(bc)-a(bc)-\alpha{ac}-\gamma{ab}-\psi{a}=(\beta-1+\gamma+\alpha-2\psi)a$, as required.
Clearly, the cases $x=b$ and $x=c$  are similar to this one. 

Since $ae=\lambda{a}$, we have $e\in A_1(a)+A_0(a)$. It follows from Lemma~\ref{l:seress} that $(ab){e}=a(be)=a(\lambda{b})=\lambda{ab}$. Similarly, we have $e(bc)=\lambda{bc}$ and $e(ac)=\lambda{ac}$.
Now $(a(bc))e=a((bc)e)=a(\lambda{bc})=\lambda{a(bc)}$.
In the same way, we have $e(b(ac))=\lambda(b(ac))$ and $e(c(ab))=\lambda(c(ab))$, and the lemma follows.
\end{proof}

\begin{prop}\label{p:rad} Consider the Gram matrix $G$ for the basis $a$, $b$, $c$, $ab$, $bc$,  $ac$, $abc$, $bac$, $cab$. Then $$\det(G)=\frac{1}{1024}(\alpha\beta\gamma-\psi^2)^3(\alpha+\beta+\gamma-2\psi-1)^6.$$ 
Moreover, if $R=A(\alpha,\beta,\gamma,\psi)^\perp$ is the radical of the form then 
it has a basis as in Table~\ref{t:gram} depending on the values of $\alpha$,$\beta$,$\gamma$,$\psi$.
\end{prop}
\begin{proof} We find elements of the matrix $G$ in Proposition~\ref{p:find_gram} and now it is easy to evaluate its determinant in GAP, getting the value $\frac{1}{1024}(\alpha\beta\gamma-\psi^2)^3(\alpha+\beta+\gamma-2\psi-1)^6$. 
If it is non-zero then $R$ is zero. Now the determinant is zero if and only if at least one of the expressions 
$\alpha\beta\gamma-\psi^2$ or $\alpha+\beta+\gamma-2\psi-1$ is zero.
Therefore, we need to consider several cases. If $\alpha\beta\gamma=\psi^2$
then $\alpha+\beta+\gamma$ either equal to $2\psi+1$ or not.
Moreover, we consider separately cases $\psi=0$ and $\psi=\alpha=\beta=\gamma=\beta=1$.
Finally, if $\alpha\beta\gamma\neq\psi^2$ then $\alpha+\beta+\gamma=2\psi+1$ and hence at least one of the scalars $\alpha$, $\beta$, or $\gamma$ is not equal to one. We suppose then that $\alpha\neq1$ and use it to find a non-zero minor.

For each case we provide a basis of the radical. We put information in Table~\ref{t:radical}. It is easy to see that in every case the vectors are linearly independent since each vector has a non-zero coordinate which is zero for the other vectors. To show that vectors form a basis we write a non-zero minor of required size in the fourth column.
There we denote by $G_{i_1,i_2,\ldots,i_k}$ the submatrix of $G$ which is the intersection of $i_1$-th, $i_2$-th,$\ldots$, $i_k$-th rows with 
$i_1$-th, $i_2$-th,$\ldots$, $i_k$-th columns.

\begin{table}
\begin{center}
\begingroup
\setlength{\tabcolsep}{20pt} %
\renewcommand{\arraystretch}{1.3}

{\tiny
$\begin{tabu}[h!]{|c|c|c|c|}
\hline
\text{Condition} & \text{Rank} & \text{Basis of radical} & \text{Non-zero minor} \\ \hline
\begin{tabu}{@{}c@{}} \alpha\beta\gamma=\psi^2, \psi\neq0, \\ \alpha+\beta+\gamma\neq2\psi+1 \end{tabu} &  6 &
\begin{tabu}{@{}c@{}} (0, 0, 0, -\beta\gamma, 0, -\alpha\beta, 2\psi, 0, 0), \\ (0, 0, 0, -\beta\gamma, -\alpha\gamma, 0, 0, 2\psi, 0), \\ (0, 0, 0, 0, -\alpha\gamma, -\alpha\beta, 0, 0, 2\psi) \end{tabu}  & \begin{tabu}{@{}c@{}} \det(G_{1,2,3,4,5,6})=\\=\frac{\psi^2}{8}\cdot(\alpha+\beta+\gamma-2\psi-1)^4 \end{tabu} \\ \hline
\begin{tabu}{@{}c@{}} \alpha\beta\gamma=\psi^2, \psi\neq0,\alpha\neq1, \\ \alpha+\beta+\gamma=2\psi+1 \end{tabu} & 3 &
\begin{tabu}{@{}c@{}} 
(\alpha(\beta-1), \alpha(\gamma-1), \alpha(1-\alpha), 2\alpha-2\psi, 0, 0, 0, 0, 0 ), \\
(0, \alpha\beta-\alpha\psi, 0, \psi-\alpha\beta, \alpha^2-\alpha, 0, 0, 0, 0), \\
(\alpha\gamma-\alpha\psi, 0, 0, \psi-\alpha\gamma, 0, \alpha^2-\alpha, 0, 0, 0), \\
(\alpha\psi-\alpha^2\beta, 0, 0, \alpha+\psi-\alpha^2-\alpha\gamma, 0, 0, 2\alpha^2-2\alpha, 0, 0), \\
(0, \alpha\psi-\alpha^2\gamma, 0, \alpha+\psi-\alpha^2-\alpha\beta, 0, 0, 0, 2\alpha^2-2\alpha, 0 ), \\
(\psi-\alpha\beta, \psi-\alpha\gamma, 0, 1-\alpha, 0, 0, 0, 0, 2\alpha-2 )
\end{tabu} &
det(G_{1,2,4})= -\frac{\alpha}{2}(\alpha-1)^3  \\ \hline
\psi=\alpha=\beta=\gamma=1  & 1 & 
\begin{tabu}{@{}c@{}}
(-1, 1, 0, 0, 0, 0, 0, 0, 0), (-1, 0, 1, 0, 0, 0, 0, 0, 0), \\
(-1, 0, 0, 1, 0, 0, 0, 0, 0), (-1, 0, 0, 0, 1, 0, 0, 0, 0), \\
(-1, 0, 0, 0, 0, 1, 0, 0, 0), (-1, 0, 0, 0, 0, 0, 1, 0, 0), \\
(-1, 0, 0, 0, 0, 0, 0, 1, 0), (-1, 0, 0, 0, 0, 0, 0, 0, 1 )
\end{tabu}
& \det(G_{1})=1\\ \hline
\psi=\alpha=\beta=\gamma=0 & 3 & 
\begin{tabu}{@{}c@{}} 
(0, 0, 0, 1, 0, 0, 0, 0, 0),  (0, 0, 0, 0, 1, 0, 0, 0, 0), \\ 
(0, 0, 0, 0, 0, 1, 0, 0, 0),  (0, 0, 0, 0, 0, 0, 1, 0, 0), \\ 
(0, 0, 0, 0, 0, 0, 0, 1, 0), (0, 0, 0, 0, 0, 0, 0, 0, 1) 
\end{tabu}
& \det(G_{1,2})=1 \\ \hline
\psi=\alpha=\beta=0, \gamma\neq0,1 & 4 &
\begin{tabu}{@{}c@{}} (0, 0, 0, 1, 0, 0, 0, 0, 0),  (0, 0, 0, 0, 1, 0, 0, 0, 0), \\ (0, 0, 0, 0, 0, 0, 1, 0, 0),  (0, 0, 0, 0, 0, 0, 0, 1, 0), \\ (0, 0, 0, 0, 0, 0, 0, 0, 1) 
\end{tabu} 
& \begin{tabu}{@{}c@{}} \det(G_{1,2,3,6})=\\=-1/2(\gamma^4-3\gamma^3+3\gamma^2-\gamma)=\\=-1/2(\gamma-1)^3\cdot\gamma
\end{tabu} \\ \hline
\psi=\alpha=\beta=0, \gamma=1 & 2 & 
\begin{tabu}{@{}c@{}} 
(-1, 0, 1, 0, 0, 0, 0, 0, 0), (0, 0, 0, 1, 0, 0, 0, 0, 0), \\ 
(0, 0, 0, 0, 1, 0, 0, 0, 0), (-1, 0, 0, 0, 0, 1, 0, 0, 0), \\
(0, 0, 0, 0, 0, 0, 1, 0, 0), (0, 0, 0, 0, 0, 0, 0, 1, 0), \\ 
(0, 0, 0, 0, 0, 0, 0, 0, 1) \end{tabu}
  & \det(G_{1,2})=1 \\ \hline
\begin{tabu}{@{}c@{}}
\psi=\alpha=0, \beta,\gamma\neq0, \\
\beta+\gamma=1 
\end{tabu}
& 3 & 
\begin{tabu}{@{}c@{}}
(0, 0, 0, 1, 0, 0, 0, 0, 0), \\ (\cfrac{1}{2}\gamma, -\cfrac{1}{2}\beta, -\cfrac{1}{2}, 0, 1, 0, 0, 0, 0), \\ (-\cfrac{1}{2}\gamma, \cfrac{1}{2}\beta, -\cfrac{1}{2}, 0, 0, 1, 0, 0, 0), \\ 
(\cfrac{1}{4}\gamma, \cfrac{1}{4}\beta, -\cfrac{1}{4}, 0, 0, 0, 1, 0, 0), \\ (\cfrac{1}{4}\gamma, \cfrac{1}{4}\beta, -\cfrac{1}{4}, 0, 0, 0, 0, 1, 0), \\ (0, 0, 0, 0, 0, 0, 0, 0, 1) \end{tabu}
 & \det(G_{1,2,3})=2\beta(1-\beta)
 \\ \hline
\begin{tabu}{@{}c@{}}
\psi=\alpha=0, \beta,\gamma\neq0, \\
\beta+\gamma\neq1 
\end{tabu}
& 6 & 
\begin{tabu}{@{}c@{}}
(0, 0, 0, 1, 0, 0, 0, 0, 0), \\ 
(0, 0, 0, 0, 0, 0, -1, 1, 0), \\ 
(0, 0, 0, 0, 0, 0, 0, 0, 1)
\end{tabu} & 
\begin{tabu}{@{}c@{}}
\det(G_{1,2,3,5,6,7})=\\
=\cfrac{1}{32}\gamma^2\beta^2(\beta+\gamma-1)^4
\end{tabu}
 \\ \hline
\begin{tabu}{@{}c@{}}
\psi^2\neq\alpha\beta\gamma, \\
\alpha+\beta+\gamma=2\psi+1, \\ \alpha\neq1 
\end{tabu} & 4 &
\begin{tabu}{@{}c@{}}
(\frac{1}{2}(\beta-1), \frac{1}{2}(\beta-\alpha), \frac{1}{2}(1-\alpha), 1-\beta, \alpha-1, 0, 0, 0, 0 ),\\
(\frac{1}{2}(\gamma-\alpha), \frac{1}{2}(\gamma-1), \frac{1}{2}(1-\alpha), 1-\gamma,  0, \alpha-1, 0, 0, 0),\\
(2\psi-2\alpha\beta+\beta-1, \gamma-1, 1-\alpha, 4-2\alpha-2\gamma,0,0,4\alpha-4,0,0), \\
(\beta-1, 2\psi-2\alpha\gamma+\gamma-1, 1-\alpha,4-2\alpha-2\beta,0,0,0,4\alpha-4,0), \\
(\psi-\alpha, \psi-\alpha, \alpha(1-\alpha), 2-\beta-\gamma, 0, 0, 0, 0, 2\alpha-2)
\end{tabu}
& \begin{tabu}{@{}c@{}}
det(G_{1,2,3,4})=\\=(\alpha-1)^2(\alpha\beta\gamma-\psi^2)
\end{tabu}
\\ \hline
\end{tabu}$} 

\caption{Bases of the radical}\label{t:radical}
\endgroup
\end{center}
\end{table}

\end{proof}

\subsection{Proof of Theorem~\ref{thm:2}}

In this subsection we prove Theorem~\ref{thm:2}. 
By Proposition~\ref{p:rad}, the radical of the Frobenius form is zero if and only if 
$(\alpha+\beta+\gamma-2\psi-1)(\alpha\beta\gamma-\psi^2)=0$.
Therefore, it remains to prove claims $(i)$ and $(ii)$ of the theorem.
The proof is split into the following two statements.

\begin{prop}\label{p:M3F} Suppose that $(\alpha+\beta+\gamma-2\psi-1)(\alpha\beta\gamma-\psi^2)\neq0$ and   $\psi^2-\alpha\beta\gamma$ is a square in $\mathbb{F}$. Then the algebra $A(\alpha,\beta,\gamma,\psi)$
is isomorphic to $M_3(\mathbb{F})^+$.
\end{prop}
\begin{proof}
Consider the following matrices in $M_3(\mathbb{F})$:
$$A=\left( \begin{matrix} 1 & 0 & 0 & \\ \alpha & 0 & 0 \\ t & 0 & 0\end{matrix}\right), B=\left( \begin{matrix} 0 & 1 & 0 & \\ 0 & 1 & 0 \\ 0 & \beta & 0\end{matrix}\right), C=\left( \begin{matrix} 0 & 0 & \gamma/t & \\ 0 & 0 & 1 \\ 0 & 0 & 1\end{matrix}\right),$$ where $t$ is a parameter which we define later.
It is easy to see that $A^2=A$, $B^2=B$ and $C^2=C$.
Now $$A\circ B=\left(\begin{matrix} \frac{\alpha}{2} & \frac{1}{2} & 0 \\ \frac{\alpha}{2} & \frac{\alpha}{2} & 0 \\ \frac{\alpha\beta}{2} & \frac{t}{2} & 0 \end{matrix}\right), B\circ C=\left(\begin{matrix}
0 & \frac{\beta\gamma}{2t} & \frac{1}{2} \\
0 & \frac{\beta}{2} & \frac{1}{2} \\
0 & \frac{\beta}{2} & \frac{\beta}{2}
\end{matrix}\right), A\circ C=\left(\begin{matrix}
\frac{\gamma}{2} & 0 & \frac{\gamma}{2t} \\
\frac{t}{2} & 0 & \frac{\alpha\gamma}{2t} \\
\frac{t}{2} & 0 & \frac{\gamma}{2}
\end{matrix}\right)$$
Moreover, we see that
$$A\circ(B\circ C)=\left(\begin{matrix}
\frac{\alpha\beta\gamma+t^2}{4t} & \frac{\beta\gamma}{4t} & \frac{1}{4} \\
\frac{\alpha\beta+t}{4} & \frac{\alpha\beta\gamma}{4t} & \frac{\alpha}{4} \\
\frac{\alpha\beta+\beta{t}}{4} & \frac{\beta\gamma}{4} & \frac{t}{4}
\end{matrix}\right)
, B\circ(A\circ C)=\left(\begin{matrix}
\frac{t}{4} & \frac{\beta\gamma+t\gamma}{4t} & \frac{\alpha\gamma}{4t} \\
\frac{t}{4} & \frac{\alpha\beta\gamma+t^2}{4t} & \frac{\alpha\gamma}{4t} \\
\frac{\beta{t}}{4} & \frac{\beta\gamma+t}{4} & \frac{\alpha\beta\gamma}{4t}
\end{matrix}\right),$$  $$C\circ(A\circ B)=\left(\begin{matrix}
\frac{\alpha\beta\gamma}{4t} & \frac{\gamma}{4} & \frac{\alpha\gamma+t}{4t} \\
\frac{\alpha\beta}{4} & \frac{t}{4} & \frac{\alpha{t}+\alpha\gamma}{4t} \\
\frac{\alpha\beta}{4} & \frac{t}{4} & \frac{\alpha\beta\gamma+t^2}{4t}
\end{matrix}\right).$$

It is known that a map $(\cdot,\cdot):M_3(\mathbb{F})^2\rightarrow\mathbb{F}$ such that $(X, Y)=tr(XY)=tr(X\circ Y)$, where $X,Y\in M_3(\mathbb{F})$, is a symmetric bilinear form on $M_3(\mathbb{F})$. It is easy to see that this form associates with the product $\circ$.
Clearly, we have $tr(A\circ A)=tr(B\circ B)=tr(C\circ C)=1$.

Furthermore, we see that $tr(A\circ B)=\alpha$, $tr(B\circ C)=\beta$, $tr(A\circ C)=\gamma$
and $tr(A, B\circ C)=tr(B, A\circ C)=tr(C, A\circ B)=\frac{2\alpha\beta\gamma+2t^2}{4t}$. 
At least one of elements $\psi\pm\sqrt{\psi^2-\alpha\beta\gamma}$ is non-zero and
we denote this element by $t$. We claim that $\frac{2\alpha\beta\gamma+2t^2}{4t}$ equals $\psi$. Clearly, 
$t$ is a root of the quadratic equation $x^2-2x\psi+\alpha\beta\gamma=0$ and hence $2t^2+2\alpha\beta\gamma=4t\psi$. This implies that $\frac{2\alpha\beta\gamma+2t^2}{4t}=\psi$.
Consider the matrix $N$ whose rows are $A$, $B$, $t\cdot C$, $A\circ B$, $t\cdot B\circ C$, $t\cdot A\circ C$, $t\cdot A\circ(B\circ C)$, $t\cdot B\circ(A\circ C)$, $t\cdot C\circ(A\circ B)$.
Using GAP, we see that $\det(4\cdot N)=1024\cdot(\alpha\beta\gamma-t^2)^3(\alpha\beta\gamma-\alpha{t}-\beta{t}-\gamma{t}+t^2+t)^3$. Now we show that $N$ is non-degenerate. 
Suppose that $\alpha\beta\gamma-\alpha{t}-\beta{t}-\gamma{t}+t^2+t=0$. Since $t^2=2t\psi-\alpha\beta\gamma$,
we obtain $0=-\alpha{t}-\beta{t}-\gamma{t}+2t\psi+t=-t(\alpha+\beta+\gamma-2\psi-1)$. However, we know that $t\neq0$ and $\alpha+\beta+\gamma-2\psi-1\neq0$; a contradiction. Assume now that $\alpha\beta\gamma-t^2=0$. Since $t^2=2t\psi-\alpha\beta\gamma$, we have $t\psi =\alpha\beta\gamma$. Then $\alpha\beta\gamma\neq0$.
Squaring up the equality $t\psi-\psi^2 =\alpha\beta\gamma-\psi^2$, we find that $\psi^2(t-\psi)^2=(\alpha\beta\gamma-\psi^2)^2$. We know that $t=\psi\pm\sqrt{\psi^2-\alpha\beta\gamma}$ and hence $\psi^2(t-\psi)^2=\psi^2(\psi^2-\alpha\beta\gamma)$. This implies $\psi^2=\psi^2-\alpha\beta\gamma$ and, therefore, $\alpha\beta\gamma=0$; a contradiction.
Thus $N$ is non-degenerate and hence elements $A$, $B$, $C$, $A\circ B$, $B\circ C$, $A\circ C$, $A\circ(B\circ C)$, $ B\circ(A\circ C)$, $C\circ(A\circ B)$ generate $M_3(\mathbb{F})$ as a vector space.
Now we find the 1-eigenspace of $ad_A$. Suppose that $A\circ X=X$, where $X=(x_{ij})$.
Then $$0=AX+XA-2X=\left(\begin{matrix}
ax_{12}+tx_{13} & -x_{12} & -x_{13} \\
ax_{11}-x_{21}+ax_{22}+tx_{23} & ax_{12}-2x_{22} & ax_{13}-2x_{23} \\
tx_{11}-x_{31}+ax_{32}+tx_{33} & tx_{12}-2x_{32} & tx_{13}-2x_{33}
\end{matrix}\right).$$
Therefore, we see that $x_{12}=x_{13}=x_{22}=x_{23}=x_{33}=x_{32}=0$
and $x_{21}=ax_{11}$, $x_{31}=tx_{11}$. So we obtain $X=x_{11}A$ and the eigenspace is 1-dimensional. Similarly, we find that 1-eigenspaces for $ad_B$ and $ad_C$ are 1-dimensional. Since $M_3(\mathbb{F})^+$ is a Jordan algebra it is an axial algebra of Jordan type half. Moreover, we prove that $A$, $B$, $C$ are primitive axes that generate $M_3(\mathbb{F})^+$. Now the table of their products is as in Table~\ref{t:prod} by Proposition~\ref{p:span} and hence $M_3(\mathbb{F})^+$ is isomorphic to $A(\alpha,\beta,\gamma,\psi)$. 
\end{proof}

If $\psi^2-\alpha\beta\gamma$ is not a square in $\mathbb{F}$ then
denote by $\mathbb{P}$ the field $\mathbb{F}(z)$, where $z^2=\psi^2-\alpha\beta\gamma$,
and let the bar denote the involution of $\mathbb{P}$ given by $x+zy\mapsto\overline{x+zy}= x-zy$ for $x,y\in\mathbb{F}$. The transpose of a matrix $T$ is denoted by $T'$.

\begin{prop} Suppose that $(\alpha+\beta+\gamma-2\psi-1)(\alpha\beta\gamma-\psi^2)\neq0$ and   $\psi^2-\alpha\beta\gamma$ is not a square in $\mathbb{F}$. 
Then $A$ is isomorphic to $H(M_3(\mathbb{P}),j)$, where $j$ is the involution given by $X^j=T^{-1}\overline{X}'T$ with some $T\in M_3(\mathbb{P})$ such that $T^j=tT$, $t\in\mathbb{P}$ and $t\overline{t}=1$.
\end{prop}
\begin{proof}
Denote by $A$ the algebra $A(\alpha,\beta,\gamma,\psi)$ and by $A_\mathbb{P}$ the algebra $A\otimes_\mathbb{F}\mathbb{P}$, i.e. the $\mathbb{P}$-algebra with the table of products as in Table~\ref{t:prod}. By Proposition~\ref{p:M3F}, we infer that $A_\mathbb{P}$ is isomorphic to $M_3(\mathbb{P})^+$
and we assume that $A_\mathbb{P}=M_3(\mathbb{P})^+$.
The product on $M_3(\mathbb{P})$ is denoted by $\cdot$. Then $xy=\frac{1}{2}(x\cdot y+y\cdot x)$.
Consider the map $j:A_\mathbb{P}\rightarrow{A}_\mathbb{P}$ such that
$(x\otimes1+y\otimes z)^j=x\otimes1-y\otimes z$. Then $j$ is an automorphism of the Jordan $\mathbb{F}$-algebra $A_{\mathbb{P}}$. It follows from \cite{Hir,Sml} that $j$ is either an automorphism of the matrix $\mathbb{F}$-algebra $M_3(\mathbb{P})$ or its anti-automorphism.

Suppose that $j$ is an automorphism of $M_3(\mathbb{P})$ over $\mathbb{F}$.
Consider $u=x\otimes1$, $v=y\otimes1$. Then $u\cdot v=e\otimes1+f\otimes z$, where $e,f\in A$. Therefore, $(u\cdot v)^j=e\otimes1-f\otimes z$. Since $(u\cdot v)^j=u^j\cdot v^j=u\cdot v$, we have $e\otimes1+f\otimes z=e\otimes1-f\otimes z$
and hence $f=0$, $u\cdot v=e\otimes1$. Thus $A\otimes 1$ is an associative algebra with respect to the product $\cdot$, and $xy=\frac{1}{2}(x\cdot y+y\cdot x)$ for $x,y\in A\otimes 1$. By Proposition~\ref{p:summary} and Lemma~\ref{l:identity}, $A\otimes 1$ is a unital simple algebra. If $z$ is an element from its center then $za\in A_1(a)$ and $z$ is inverible. 
Hence $za=\alpha a$ with $z\in\mathbb{F}$ and, therefore, $Z(A\otimes 1)=\mathbb{F}$. By the Artin-Wedderburn theorem $A$ is isomorphic to $M_3(\mathbb{F})^+$, since $\dim_FA=9$. By Proposition~\ref{p:rad}, the determinant of the Gram matrix of $A$
is $\frac{1}{1024}(\alpha\beta\gamma-\psi^2)^3(\alpha+\beta+\gamma-2\psi-1)^6$ which is $-(\psi^2-\alpha\beta\gamma)x^2$, where $x\in\mathbb{F}$. On the other hand, the Frobenius form on $X, Y\in M_3(\mathbb{F})^+$ equals to $tr(X\circ Y)$. Then the determinant of the Gram matrix for the standard basis of matrix units equals $-1$. Since $\psi^2-\alpha\beta\gamma$
is not a square in $\mathbb{F}$, we arrive at a contradiction.

Suppose that $j$ is an anti-automorphism of $M_3(\mathbb{P})$ over $\mathbb{F}$.
Then $j$ is also an involution of this algebra such that $A\otimes1=H(M_3(\mathbb{P}),j)$.
So we can assume that $A=H(M_3(\mathbb{P}),j)$. Now we describe $j$.
The map $j_1:M_3(\mathbb{P})\rightarrow M_3(\mathbb{P})$ with $j_1:X\mapsto\overline{X}'$ is an involution of $M_3(\mathbb{P})$. Since $z^{jj_1}=z$, the product $\phi=jj_1$ is an automorphism of $M_3(\mathbb{P})$. All $\mathbb{P}$-automorphism of $M_3(\mathbb{P})$ are inner, hence $\phi(X)=X^{jj_1}=TXT^{-1}$,
where $T\in M_3(\mathbb{P})$. Then $X^j=(TXT^{-1})^{j_1}=(T^{-1})^{j_1}X^{j_1}T^{j_1}$ and hence $X=(X^j)^j=((T^{-1})^{j_1}X^{j_1}T^{j_1})^j=(T^{-1})^{j_1}TXT^{-1}T^{j_1}$.
Then we have $T^{-1}T^{j_1}=t\cdot I$, where $t\in\mathbb{P}$ and $T^{j_1}=tT$.
Therefore $T=(T^{j_1})^{j_1}=(tT)^{j_1}=\overline{t}tT$.
Thus $\overline{t}t=1$ and hence $X^j=T^{-1}X^{j_1}T$. Finally we see that $T^j=T^{j_1}=tT$.
\end{proof}

\section*{Acknowledgements}
The authors would like to thank Prof. Sergey Shpectorov for reading a draft of this manuscript and for his valuable comments and suggestions on the text. The authors are grateful to Prof. Victor Zhelyabin for discussions of Jordan algebras and Theorem~\ref{thm:2}.

\Addresses
\end{document}